\documentclass[11pt]{amsart}

\title[mapping class groups]{The mapping class groups of reducible Heegaard splittings of genus two}

\author{Sangbum Cho}
\thanks{The first-named author is supported by Basic Science Research Program through the National Research Foundation of Korea (NRF) funded by the Ministry of Science, ICT and Future Planning (NRF-2015R1A1A1A05001071).}
\address{Department of Mathematics Education  \newline
\indent Hanyang University, Seoul 133-791, Korea}
\email{scho@hanyang.ac.kr}

\author{Yuya Koda}
\thanks{The second-named author is supported in part
by the Grant-in-Aid for Young Scientists (B), JSPS KAKENHI Grant Number 26800028.}

\address{
Department of Mathematics \newline
\indent Hiroshima University, 1-3-1 Kagamiyama, Higashi-Hiroshima, 739-8526, Japan}
\email{ykoda@hiroshima-u.ac.jp}

\subjclass[2000]{Primary 57N10; 57M60.}

\date{\today}

\usepackage{amsfonts,amsmath,amssymb,amscd}

\usepackage{amsthm}
\usepackage{latexsym}
\usepackage[dvips]{graphicx}
\usepackage[dvips]{psfrag}
\usepackage[dvips]{color}
\usepackage{xypic}
\usepackage[abs]{overpic}
\usepackage{caption}
\usepackage[all]{xy}
\usepackage{xypic}
\usepackage[all]{xy}

\theoremstyle{plain}
\newtheorem*{theorem*}{Theorem}
\newtheorem*{lemma*} {Lemma}
\newtheorem*{corollary*} {Corollary}
\newtheorem*{proposition*}{Proposition}
\newtheorem*{conjecture*}{Conjecture}
\newtheorem{theorem}{Theorem}[section]
\newtheorem{lemma}[theorem]{Lemma}

\theoremstyle{remark}

\theoremstyle{definition}

\newcommand{\Natural}{\mathbb{N}}
\newcommand{\Integer}{\mathbb{Z}}
\newcommand{\Rational}{\mathbb{Q}}

\newcommand{\MCG}{\mathrm{MCG}}

\newcommand{\Diff}{\mathrm{Diff}}

\newcommand{\Nbd}{\operatorname{Nbd}}
\newcommand{\cl}{\operatorname{cl}}
\newcommand{\TT}{\mathcal{T}^{\mathcal{T}}
}



\begin{document}
\maketitle

\begin{abstract}
The manifold which admits a genus-$2$ reducible Heegaard splitting is one of the $3$-sphere, $\mathbb{S}^2 \times \mathbb{S}^1$, lens spaces and their connected sums.
For each of those manifolds except most lens spaces, the mapping class group of the genus-$2$ splitting was shown to be finitely presented.
In this work, we study the remaining generic lens spaces, and show that the mapping class group of the genus-$2$ Heegaard splitting is finitely presented for any lens space by giving its explicit presentation.
As an application, we show that the fundamental groups of the spaces of the genus-$2$ Heegaard splittings of lens spaces are all finitely presented.
\end{abstract}




%

\section*{Introduction}

It is well known that every closed orientable $3$-manifold $M$ can be decomposed into two handlebodies $V$ and $W$ of the same genus $g$ for some $g \geq 0$.
That is, $V \cup W = M$ and $V \cap W = \partial V = \partial W = \Sigma$, a genus-$g$ closed orientable surface.
We call such a decomposition a {\it Heegaard splitting} for the manifold $M$ and denote it by $(V, W; \Sigma)$.
The surface $\Sigma$ is called the {\it Heegaard surface} of the splitting, and the genus of $\Sigma$ is called the {\it genus} of the splitting.
The $3$-sphere admits a Heegaard splitting of each genus $g \geq 0$, and a lens space a Heegaard splitting of each genus $g \geq 1$.

The {\it mapping class group of a Heegaard splitting} for a manifold is the group of isotopy classes of orientation-preserving diffeomorphisms of the manifold that preserve each of the two handlebodies of the splitting setwise.
We call such a group the {\it Goeritz group} of the splitting, or a {\it genus-$g$ Goeritz group} when the splitting has genus $g$ as well.
Further, when a manifold admits a unique Heegaard splitting of genus-$g$ up to isotopy, we call the Goeritz group of the splitting simply the {\it genus-$g$ Goeritz group} of the manifold without mentioning a specific splitting.
We note that the mapping class group of a Heegaard splitting is a subgroup of the mapping class group of the Heegaard surface.

It is important to understand the structure of a Goeritz group,
in particular, a finite generating set or a finite presentation of it if any.
For example, using a finite presentation of the genus-$2$ Goeritz group of the 3-sphere given in \cite{Ge}, \cite{Sc}, \cite{Ak} and \cite{C},
it is constructed a new theory on the collection of the tunnel number-$1$ knots in \cite{CM09}.
Further, it has been an open problem whether the fundamental groups of the spaces of genus-$2$ Heegaard splittings of lens spaces
are finitely generated/presented or not, see \cite{JM13}.
If the genus-$2$ Goeritz groups are shown to be finitely presented, then so are those fundamental groups.

In \cite{CK14} a finite presentation of the genus-$2$ Goeritz group of $\mathbb S^2 \times \mathbb S^1$ was obtained, and in \cite{CK15a} finite presentations of the genus-$2$ Goeritz groups were obtained for the connected sums whose summands are $\mathbb S^2 \times \mathbb S^1$ or lens spaces.
We refer the reader to \cite{Joh10}, \cite{Joh11}, \cite{Sc13}, \cite{Kod} and \cite{CKA}
for finite presentations or finite generating sets of the Goeritz groups of several Heegaard splittings and related topics.

For the genus-$2$ Goeritz groups of lens spaces, finite presentations are obtained only for a small class of lens spaces in \cite{C2} and \cite{CK15b}.
That is, for the lens spaces $L(p, q)$, $1\leq q \leq p/2$, under the condition $p \equiv \pm 1 \pmod q$.
In this work, we study the remaining generic lens spaces, the case of $p \not\equiv \pm 1 \pmod q$.
We show that the genus-$2$ Goeritz group of each of those lens spaces is again finitely presented and obtain an explicit presentation, which is introduced in Theorem \ref{thm:presentations of the Goeritz groups for non-connected case} in Section \ref{sec:main_section}.
The manifold which admits a genus-$2$ reducible Heegaard splitting is one of the $3$-sphere, $\mathbb S^2 \times \mathbb S^1$, lens spaces and their connected sums.
Therefore, Theorem \ref{thm:presentations of the Goeritz groups for non-connected case} together with the previous results mentioned above implies the following.

\begin{theorem}
The mapping class group of each of the reducible Heegaard splittings of genus-$2$ is finitely presented.
\end{theorem}

In other words, the theorem says that the mapping class groups of genus-$2$ Heegaard splittings of Hempel distance $0$
are all finitely presented.
It is shown in \cite{Nam07} and \cite{Joh11} that the mapping class groups
are all finite for the Heegaard splittings of Hempel distance at least $4$.
The mapping class groups of the splittings of Hempel distances $2$ and $3$ still remain mysterious.
(Here note that there are no genus-$2$ splittings of Hempel distance $1$.)

To obtain a presentation of the Goeritz group, we have constructed a simply connected simplicial complex on which the group acts ``nicely'', in particular, so that the quotient of the action is a simple finite complex. And then we calculate the isotropy subgroups of each of the simplices of the quotient, and express the Goeritz group in terms of those subgroups.

For the genus-$2$ Heegaard splitting $(V, W; \Sigma)$ of a lens space $L(p, q)$ with $1\leq q \leq p/2$, we have constructed the {\it primitive disk complex}, denoted by $\mathcal P(V)$, whose vertices are defined to be the isotopy classes of the {\it primitive disks} in the handlebody $V$.
In \cite{CK15b}, the combinatorial structure of the complex $\mathcal P(V)$ are fully studied and it was shown that $\mathcal P(V)$ is simply connected, in fact contractible, under the condition $p \equiv \pm 1 \pmod q$, and is used to obtained the presentation of the Goeritz group.
In the case of $p \not\equiv \pm 1 \pmod q$, the complex $\mathcal P(V)$ is no longer simply connected. In fact, it consists of infinitely many tree components isomorphic to each other.
In the present paper, we will construct a new simplicial complex for this case, which we will call the {\it ``tree of trees''}, whose vertices are the tree components of $\mathcal P(V)$.

In Section \ref{sec:disk_complex}, it will be briefly reviewed the primitive disk complex $\mathcal P(V)$ for the genus-$2$ Heegaard splitting of each lens space.
In Section \ref{sec:tree_of_trees}, we construct the complex ``tree of trees'' for the case of $p \not\equiv \pm 1 \pmod q$ and develop some related properties that we need.
In the main section, Section \ref{sec:main_section},  the action of the Goeritz group on the tree of trees will be investigated to obtain the presentation of the group.
Right before Section \ref{sec:main_section}, the simplest example of our case $p \not\equiv \pm 1 \pmod q$, the lens space $L(12, 5)$, will be studied in detail in Section \ref{sec:first_example} as a motivating example.
In the final section, we show that the fundamental groups of the spaces of genus-$2$ Heegaard splittings of lens spaces are all finitely presented (up to the Smale Conjecture for $L(2,1)$).

We use the standard notation $L = L(p, q)$ for a lens space.
We refer \cite{Ro} to the reader.
The integer $p$ can be assumed to be positive.
It is well known that two lens spaces $L(p, q)$ and $L(p', q')$ are diffeomorphic if and only if $p = p'$ and  $q'q^{\pm 1} \equiv \pm 1 \pmod p$.
Thus, we will assume $1 \leq q \leq p/2$ for the lens space $L(p, q)$.
Note that each lens space admits a unique Heegaard splitting of each genus $g \geq 1$ up to isotopy by \cite{BO83}.
Throughout the paper, any disks in a handlebody are always assumed to be properly embedded, and their intersection is transverse and minimal up to isotopy.
In particular, if a disk $D$ intersects a disk $E$, then $D \cap E$ is a collection of pairwise disjoint arcs that are properly embedded in both $D$ and $E$.
For convenience, we will not distinguish disks (or union of disks) and diffeomorphisms from their isotopy
classes in their notation.
Finally, $\Nbd(X)$ will denote a regular neighborhood of $X$ and $\cl(X)$ the closure of $X$ for a subspace $X$ of a space, where the ambient space will always be clear from the context.

\section{The primitive disk complexes}
\label{sec:disk_complex}

\subsection{The non-separating disk complex for the genus-$2$ handlebody}

Let $V$ be a genus-$2$ handlebody.
The {\it non-separating disk complex}, denoted by $\mathcal{D}(V)$, of $V$ is a simplicial complex whose vertices are the isotopy classes of non-separating disks in $V$ such that a collection of $k+1$ vertices spans a $k$-simplex if and only if it admits a collection of representative disks which are pairwise disjoint.
We note that the disk complex $\mathcal{D}(V)$ is $2$-dimensional and every edge of $\mathcal D(V)$ is contained in infinitely but countably many $2$-simplices.
In \cite{McC}, it is proved that $\mathcal D(V)$ and the link of any vertex of $\mathcal D(V)$ are all contractible.
Thus, the complex $\mathcal D(V)$ deformation retracts to a tree in its barycentric subdivision spanned by the barycenters of the $1$-simplices and $2$-simplices, which we call the {\it dual tree} of $\mathcal{D}(V)$.
See Figure \ref{fig:disk_complex}.
We note that each component of any full subcomplex of $\mathcal{D}(V)$ is contractible.

\begin{figure}[htbp]
\begin{center}
\includegraphics[width=9cm,clip]{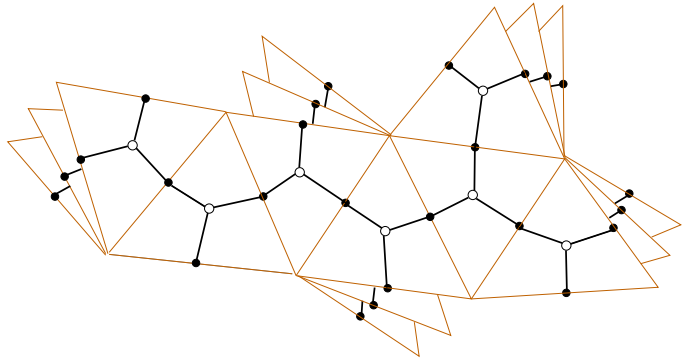}
\caption{A portion of the non-separating disk complex $\mathcal D(V)$ of a genus-$2$ handlebody $V$ with its dual tree.}
\label{fig:disk_complex}
\end{center}
\end{figure}

Let $D$ and $E$ be non-separating disks in $V$ and suppose that the vertices of the disks in $\mathcal{D}(V)$, which we denote by $D$ and $E$ again, are not adjacent to each other, that is, $D \cap E \neq \emptyset$.
In the barycentric subdivision of $\mathcal{D}(V)$, the links of the vertices $D$ and $E$ are disjoint trees.
Then there exists a unique shortest path in the dual tree of $\mathcal D(V)$ connecting the two links.
Let $v_1,$ $w_1$, $v_2$, $w_2, \ldots, w_{n-1}$, $v_n$ be the sequence of vertices of this path.
We note that each $v_i$ is trivalent while each $w_i$ has infinite valency in the dual tree.
Let $\Delta_1$, $\Delta_2, \ldots, \Delta_n$ be the 2-simplices of $\mathcal{D}(V)$ whose baricenters are the trivalent vertices $v_1$, $v_2, \ldots, v_n$ respectively.
We call the full subcomplex of $\mathcal{D}(V)$ spanned by the vertices of $\Delta_1$, $\Delta_2, \ldots, \Delta_n$
the {\it corridor} connecting $D$ and $E$ and we denote it by $\mathcal{C}_{\{ D, E \} } = \{  \Delta_1 , \Delta_2 , \ldots, \Delta_n \}$.
See Figure \ref{fig:corridor}.
Let $E_*$ and $E_{**}$ be the two vertices of $\Delta_1$ other than $E$.
We call the pair $\{ E_* , E_{**} \}$ the {\it principal pair} of $E$ with respect to $D$ for the corridor $\mathcal{C}_{\{ D, E \} }$.

\bigskip

\begin{center}
\begin{overpic}[width=12cm,clip]{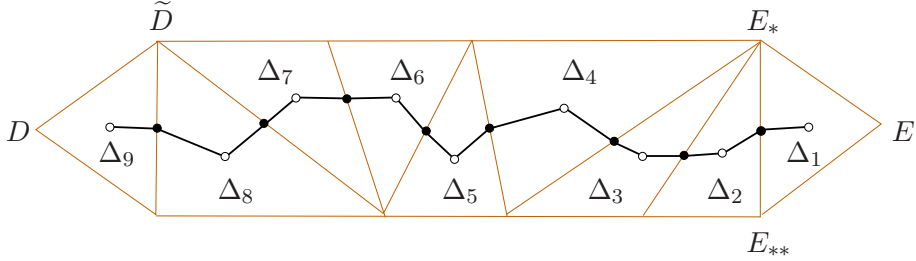}
  \linethickness{3pt}
  \put(0, 52){$D$}
  \put(54, 94){$\widetilde{D}$}
  \put(280, 94){$E_\ast$}
  \put(280, 10){$E_{\ast\ast}$}
  \put(335, 52){$E$}
  \put(295, 44){$\Delta_1$}
  \put(265, 30){$\Delta_2$}
  \put(220, 30){$\Delta_3$}
  \put(210, 75){$\Delta_4$}
  \put(165, 30){$\Delta_5$}
  \put(145, 75){$\Delta_6$}
  \put(95, 75){$\Delta_7$}
  \put(80, 30){$\Delta_8$}
  \put(35, 44){$\Delta_9$}
\end{overpic}
\captionof{figure}{The corridor connecting $D$ and $E$.}
\label{fig:corridor}
\end{center}

Let $D$ and $E$ be non-separating disks in $V$.
We assume that $D$ intersects $E$ transversely and minimally.
Let $C$ be an {\it outermost subdisk} of $D$ cut off by $D \cap E$, that is, $C$ is a disk cut off from $D$ by an arc $\alpha$ of $D \cap E$ in $D$ such that $C \cap E= \alpha$.
The arc $\alpha$ cuts $E$ into two disks, say $F_1$ and $F_2$.
Then we have two disjoint disks $E_1$ and $E_2$ which are isotopic to the disks $F_1 \cup C$ and $F_2 \cup C$ respectively.
We note that each of $E_1$ and $E_2$ is isotopic to neither $E$ nor $D$, and
each of $E_1$ and $E_2$ has fewer arcs of intersection with $D$ than $E$ had since at least the arc $\alpha$ no longer counts, as $D$ and $E$ are assumed to intersect minimally.
Further, it is easy to check that both $E_1$ and $E_2$ are non-separating, and these two disks are determined without depending on the choice of the outermost subdisk of $D$ cut off by $D \cap E$
(this is a special property of a genus-$2$ handlebody).
We call the disks $E_1$ and $E_2$ the {\it disks from surgery} on $E$ along the outermost subdisk $C$ or simply the {\it disks from surgery} on $E$ along $D$.

Let $D$, $E$ and $E_0$ be non-separating disks in $V$.
Assume that $E$ and $E_0$ are non-separating disks in $V$ which are disjoint and are not isotopic to each other, and that $D$ intersects $E \cup E_0$ transversely and minimally.
In the same way to the above, we can consider the surgery on $E \cup E_0$ along an outermost subdisk of $D$ cut off by $D \cap (E \cup E_0)$.
In fact, one of the two resulting disks from the surgery is isotopic to either $E$ or $E_0$, and the other, denoted by $E_1$, is isotopic to none of $E$ and $E_0$ (this is also a special property of a genus-$2$ handlebody).
We call the disk $E_1$ the {\it disk from surgery} on $E \cup E_0$ along $D$.
(If $D$ is already disjoint from $E \cup E_0$, then define simply $E_1$ to be $D$.)

\begin{lemma}
\label{lem:principal pair of a corridor}
Let $\mathcal{C}_{\{ D, E \} } = \{  \Delta_1 , \Delta_2 , \ldots, \Delta_n \}$ be the corridor connecting $D$ and $E$.
Then the disks of the principal pair $\{ E_*, E_{**}\}$ are exactly the disks from surgery on $E$ along $D$.
\end{lemma}
\begin{proof}
We use the induction on the number $n \geq 2$ of the $2$-simplices of the corridor.
If $n = 2$, the conclusion holds immediately since each of $D$ and $E$ is disjoint from $E_* \cup E_{**}$, and so any outermost subdisk of $D$ cut off by $D \cap E$ is also disjoint from  $E_* \cup E_{**}$.
If $n \geq 3$, choose a vertex, say $\widetilde{D}$, of $\Delta_n$ other than $D$ that is not adjacent to $E$.
Then we have the (sub)corridor  $\mathcal{C}_{\{ \widetilde{D}, E \} } = \{  \Delta_1 , \Delta_2 , \ldots, \Delta_k \}$ connecting $\widetilde{D}$ and $E$ for some $k < n$.
See Figure \ref{fig:corridor}.
By the assumption of the induction, the disks $E_*$ and $E_{**}$ are exactly the disks from surgery on $E$ along $\widetilde{D}$.
Since $D$ is disjoint from $\widetilde{D}$, the disks from surgery on $E$ along $D$ are the same to those from surgery on $E$ along $\widetilde{D}$.
\end{proof}

Consider any two consecutive $2$-simplices $\Delta_k$ and $\Delta_{k+1}$, $k \in \{1, 2, \ldots, n-1\}$, of a corridor $\mathcal{C}_{\{ D, E \} } = \{  \Delta_1 , \Delta_2 , \ldots, \Delta_n \}$ connecting $D$ and $E$.
When we write $\Delta_k = \{D_0, D_1, D_2\}$ and $\Delta_{k+1} = \{D_1, D_2, D_3\}$ as triples of vertices, we see that $\{D_1, D_2\}$ are the principal pair of $D_0$ with respect to $D$ for the (sub)corridor $\mathcal{C}_{\{ D, D_0 \} } = \{  \Delta_k , \Delta_{k+1} , \ldots, \Delta_n \}$.
Further, the followings are immediate from Lemma \ref{lem:principal pair of a corridor}.
\begin{itemize}
\item $D_1$ and $D_2$ are the disks from surgery on $D_0$ along $D$.
\item $D_3$ is the disk from surgery on $D_1 \cup D_2$ along $D$.
\end{itemize}
This observation implies the following lemma.

\begin{lemma}
\label{lem:two corridors}
Let $\mathcal{C}_{\{ D, E \} } = \{  \Delta_1 , \Delta_2 , \ldots, \Delta_n \}$ be the corridor connecting $D$ and $E$.
For each $k \in \{1, 2, \ldots, n-1\}$, we write the edge $\Delta_k \cap \Delta_{k+1} = \{D_k, D'_k\}$ and the $2$-simplex $\Delta_{k+1} = \{D_{k+1}, D_k, D'_k\}$.
Let $F$ be a vertex that is not adjacent to $E$.
If $D_{k+1}$ is the disk from surgery on $D_k \cup D'_k$ along $F$ for each $k$,
then the corridor $\mathcal{C}_{\{ F, E \} }$ connecting $F$ and $E$ contains the corridor $\mathcal{C}_{\{ D, E \} }$.
\end{lemma}

\subsection{The primitive disk complexes for lens spaces}

Let $(V, W; \Sigma)$ be the genus-$2$ Heegaard splitting of a lens space $L = L(p, q)$.
A disk $E$ properly embedded in $V$ is said to be {\it primitive} if there exists a disk $E'$ properly embedded in $W$ such that the two loops $\partial E$ and $\partial E' $ intersect transversely in a single point.
Such a disk $E'$ is called a {\it dual disk} of $E$, which is also primitive in $W$ having a dual disk $E$.
Note that both $W \cup \Nbd(E)$ and $V \cup \Nbd(E')$ are solid tori.
Primitive disks are necessarily non-separating.

The {\it primitive disk complex} $\mathcal P(V)$ for the genus-$2$ splitting $(V, W; \Sigma)$
is defined to be the full subcomplex of $\mathcal D(V)$ spanned by the primitive disks in $V$.
If a genus-$2$ Heegaard splitting admits primitive disks, then the manifold is one of the $3$-sphere, $\mathbb S^2 \times \mathbb S^1$ or a lens space, and so we can define the primitive disk complex for each of those manifolds.
The combinatorial structure of the primitive disk complexes for each of the $3$-sphere and $\mathbb S^2 \times \mathbb S^1$ has been well understood in \cite{C} and \cite{CK14}.
For the lens spaces, we have the following results from \cite{CK15b}.

\begin{theorem}[Theorems 4.2 and 4.5 in \cite{CK15b}]
\label{thm:contractibility}
For a lens space $L(p, q)$ with $1 \leq q \leq p/2$,
the primitive disk complex $\mathcal P(V)$ for
the genus-$2$ Heegaard splitting $(V, W; \Sigma)$ of $L(p, q)$ is contractible if and only if $p \equiv \pm 1 \pmod{q}$.
If $p \not\equiv \pm 1 \pmod q$, then $\mathcal P(V)$ is not connected and consists of infinitely many tree components.
\end{theorem}

In the case of $p \not\equiv \pm 1 \pmod q$, each vertex of any tree component of $\mathcal P(V)$ has infinite valency, that is, for each primitive disk $D$ in $V$ there exist infinitely many non-isotopic primitive disks disjoint from $D$.
Thus, all the tree components of $\mathcal P(V)$ are isomorphic to each other.

\section{The tree of trees}
\label{sec:tree_of_trees}

\subsection{The primitive disks}
Let $(V, W; \Sigma)$ be the genus-$2$ Heegaard splitting of a lens space $L = L(p, q)$.
In this subsection, we will develop several properties of the primitive disks in $V$ and $W$ we need, in particular, some sufficient  conditions for the non-primitiveness.
Each simple closed curve on the boundary of the genus-$2$ handlebody $W$ represents an element of
the free group $\pi_1 (W)$ of rank 2.
The following is a well known fact.

\begin{lemma}[Gordon \cite{Go}]
Let $D$ be a non-separating disk in $V$.
Then $D$ is primitive if and only if $\partial D$ represents a primitive element of $\pi_1 (W)$.
\label{lem:primitive_element}
\end{lemma}

Here, an element of a free group is said to be {\it primitive} if
it is a member of a generating set.
Primitive elements of the rank-$2$ free group has been well understood.
In particular, we have the following property.

\begin{lemma}[Osborne-Zieschang \cite{OZ}]
Given a generating pair $\{x, y\}$ of the free group $\mathbb Z \ast \mathbb Z$ of rank $2$,
a cyclically reduced form of any primitive element can be written as a product of terms each of the form
$x^\epsilon y^n$ or $x^\epsilon y^{n+1}$, or else a product of terms each of the form $y^\epsilon x^n$ or $y^\epsilon x^{n+1}$, for some $\epsilon \in \{1,-1\}$ and some $n \in \mathbb Z$.
\label{lem:property of primitive elements}
\end{lemma}

Therefore, we see that no cyclically reduced form of a primitive element
in terms of $x$ and $y$ can contain $x$ and $x^{-1}$ $($and $y$ and $y^{-1})$ simultaneously.

Let $\{E'_1, E'_2\}$ be a complete meridian system of the genus-$2$ handlebody $W$.
Assign symbols $x$ and $y$ to the oriented circles $\partial E_1'$ and $\partial E'_2$ respectively.
Then any oriented simple closed curve on $\partial W$ intersecting $\partial E_1' \cup \partial E'_2$ transversely and minimally represents an element of the free group
$\pi_1 (W) = \langle x, y \rangle$, whose word in $\{x^{\pm 1}, y^{\pm 1} \}$ can be read off from the intersections with $\partial E_1'$ and $\partial E'_2$.
Let $l$ be an oriented simple closed curve on $\partial W$ that meets $\partial E_1' \cup \partial E'_2$ transversely and minimally.
The following lemma is given in \cite{CK15b}.

\begin{lemma}[Lemma 3.3 in \cite{CK15b}]
With a suitable choice of orientations of $\partial E_1'$ and $\partial E_2'$, if a word  in $\{x^{\pm 1}, y^{\pm 1} \}$ corresponding to $l$ contains one of the pairs of terms$:$
\begin{enumerate}
\item
both of $xy$ and $xy^{-1}$, or
\item
both of $xy^nx$ and $y^{n+2}$ for $n \geq 0$,
\end{enumerate}
then the element of $\pi_1 (W)$ represented by $l$ cannot be $($a positive power of $)$ a primitive element.
\label{lem:property_of_primitive_elements2}
\end{lemma}

We introduce three more sufficient conditions for non-primitiveness as follows.

\begin{lemma}
With a suitable choice of orientations of $\partial E_1'$ and $\partial E_2'$,
if a word in $\{x^{\pm 1}, y^{\pm 1} \}$ corresponding to $l$ contains one of the pairs of terms$:$
\begin{enumerate}
\item
both of $xy$ and $(xy^{-1})^{\pm 1}$,
\item
both of $xy$ and $(x^{-1}y)^{\pm 1}$, or
\item
both of $x^2$ and $y^2$,
\end{enumerate}
then the element of $\pi_1 (W)$ represented by $l$ cannot be $($a positive power of $)$ a primitive element.
\label{lem:key}
\end{lemma}
\begin{proof}
Let $\Sigma'$ be the $4$-holed sphere cut off from $\partial W$ along $\partial E'_1 \cup \partial E'_2$.
Denote by ${e_1'}^+$ and ${e_1'}^-$ (by ${e_2'}^+$ and ${e_2'}^-$, respectively) the boundary circles of $\Sigma'$ coming from
$\partial E_1'$
(from $\partial E_2'$, respectively).

Suppose first that $l$ determines a word containing both $xy$ and $(xy^{-1})^{\pm 1}$.
We can assume that there are two subarcs $\alpha_+$ and $\alpha_-$ of $l \cap \Sigma'$ such that
$\alpha_+$ connects ${e_1'}^+$ and ${e_2'}^-$, and $\alpha_-$ connects ${e_1'}^+$ and ${e_2'}^+$ as in Figure \ref{fig:arcs_1}.

\medskip

\begin{center}
\begin{overpic}[width=3.5cm,clip]{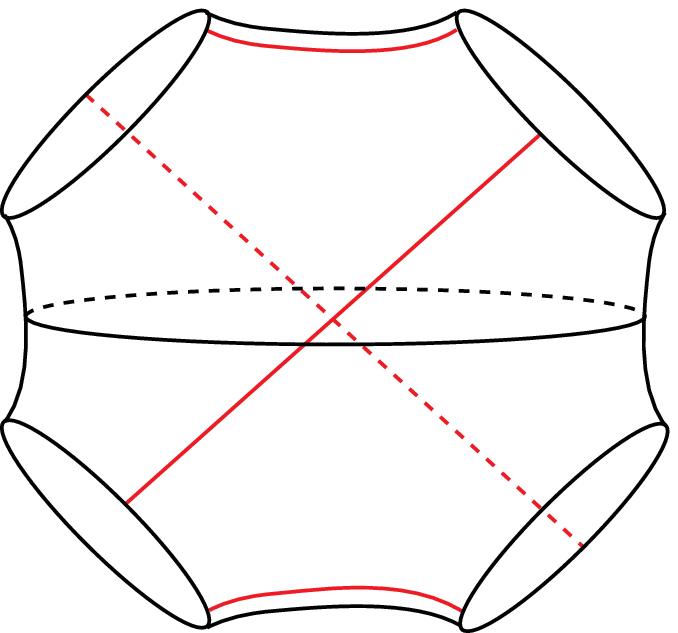}
\linethickness{3pt}
\put(-2, 80){${e_1'}^+$}
\put(-2, 3){${e_1'}^-$}
\put(90, 80){${e_2'}^-$}
\put(90, 3){${e_2'}^+$}
\put(33, 64){\color{red} \small $\alpha_-$}
\put(45, 80){\color{red} \small $\alpha_+$}
\put(33, 24){\color{red} \small $\beta_-$}
\put(45, 11){\color{red} \small $\beta_+$}
\end{overpic}
\captionof{figure}{The arcs $\alpha_\pm$ and $\beta_\pm$ on $\Sigma'$.}
\label{fig:arcs_1}
\end{center}

\medskip

Since $|~l \cap {e_1'}^+| = |~l \cap {e_1'}^-|$ and $|~l \cap {e_2'}^+| = |~l \cap {e_2'}^-|$, we must have two other arcs $\beta_+$ and $\beta_-$ of $l \cap \Sigma'$ such that $\beta_+$ connects ${e_1'}^-$ and ${e_2'}^+$, and $\beta_-$ connects ${e_1'}^-$ and ${e_2'}^-$.
See Figure \ref{fig:arcs_1}.
Consequently, there exists no arc component of $l \cap \Sigma'$ whose endpoints lines on the same boundary component of
$\Sigma'$.
That is, any word corresponding to $l$ contains neither $x^{\pm1} x^{\mp1}$ nor $y^{\pm1} y^{\mp1}$, and hence, it is cyclically reduced.
Since that word contains both
$x$ and $x^{-1}$ (or $y$ and $y^{-1}$), $l$  cannot represent (a positive power of)
a primitive element of $\pi_1 (W)$.
The case where $l$ determines a word containing both $xy$ and $(x^{-1}y)^{\pm 1}$ can be proved in the same way.

Next suppose that $l$ determines a word containing both $x^2$ and $y^2$.
Then there are two arcs $\alpha_+$ and $\alpha_-$ of $l \cap \Sigma'$ such that $\alpha_+$ connects ${e_1'}^+$ and ${e_1'}^-$, and $\alpha_-$ connects ${e_2'}^+$ and ${e_2'}^-$.
By a similar argument to the above,
we must have two other arcs $\beta_+$ and $\beta_-$ of $l \cap \Sigma'$ such that
$\beta_+$ connects ${e_1'}^+$ and ${e_2'}^\pm$, say ${e_2'}^-$, and
$\beta_-$ connects ${e_1'}^-$ and ${e_2'}^+$.
See Figure \ref{fig:arcs_2}.
\begin{center}
\begin{overpic}[width=3.5cm,clip]{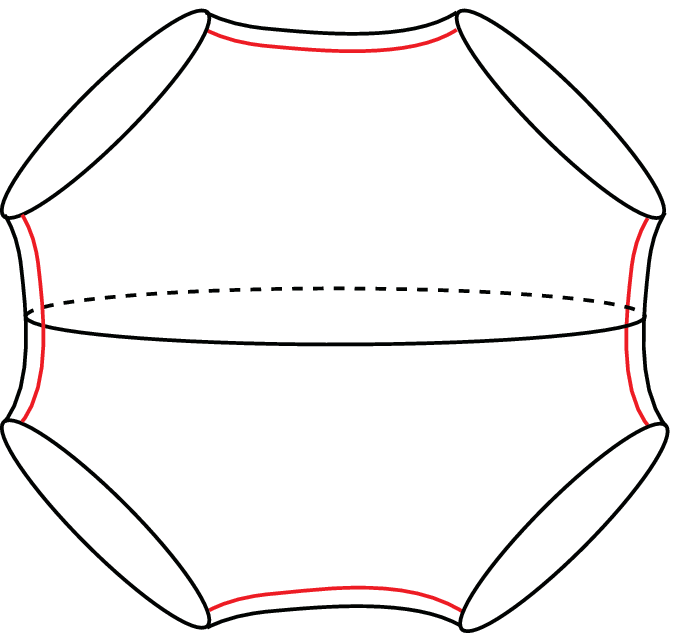}
\linethickness{3pt}
\put(-2, 80){${e_1'}^+$}
\put(-2, 3){${e_1'}^-$}
\put(90, 80){${e_2'}^-$}
\put(90, 3){${e_2'}^+$}
\put(8, 54){\color{red} \small $\alpha_+$}
\put(82, 54){\color{red} \small $\alpha_-$}
\put(45, 77){\color{red} \small $\beta_+$}
\put(45, 10){\color{red} \small $\beta_-$}
\end{overpic}
\captionof{figure}{The arcs $\alpha_\pm$ and $\beta_\pm$ on $\Sigma'$.}
\label{fig:arcs_2}
\end{center}
Then we see again that the word corresponding to $l$ is cyclically reduced.
Since that word contains both of $x^2$ and $y^2$,
$l$ cannot represent (a positive power of) a primitive element.
\end{proof}

\begin{lemma}
With a suitable choice of orientations of $\partial E_1'$ and $\partial E_2'$,
if a word in $\{x^{\pm 1}, y^{\pm 1} \}$ corresponding to $l$ contains
a term of the form
$x y^{\epsilon_{1}} y^{\epsilon_{2}} \cdots y^{\epsilon_{l}} x^{-1}$,
where $\epsilon_{i} = \pm 1$ $(i=1, 2, \ldots, l)$ and $\sum_{i=1}^{l} \epsilon_{i} \neq 0$,
then the element of $\pi_1 (W)$ represented by $l$ cannot be $($a positive power of $)$ a primitive element.
\label{lem:key2}
\end{lemma}
\begin{proof}
Set $w = x y^{\epsilon_{1}} y^{\epsilon_{2}} \cdots y^{\epsilon_{l}} x^{-1}$.
Let $\alpha$ the subarc of $l$ corresponding to the subword $w$.
By cutting the Heegaard surface $\Sigma$ along $\partial E_1' \cup \partial E_2' $,
we get a 4-holed sphere $\Sigma'$.
Let $e'^\pm_{1}$ and $e'^\pm_{2}$ be the holes coming from $E_1'$ and $E_2'$
respectively.
Without loss of generality we can assume that $\epsilon_1 =1$.
If $\epsilon_l =1$, then
we get the conclusion by Lemma \ref{lem:key}.
Thus, we assume that $\epsilon_l =-1$, this implies that the word $w$ contains the term
$y y^{-1}$.
Let $\alpha_0$, $\alpha_1$, $\alpha_2$ be
the subarcs of $\alpha$ corresponding to the term $xy$, $y y^{-1}$, $y^{-1} x^{-1}$.
Note that on the surface $\Sigma'$ the arc $\alpha_0$ connects the two circles $e'^+_1$ and $e'^-_2$,
$\alpha_1$ connects the circle $e'^-_2$ to itself, and
$\alpha_2$ connects the two circles $e'^+_2$ and $e'^-_1$.
Let $E^*$ be the band sum of $E'_1$ and $E'_2$ along $\alpha_0$, that is,
$E^*$ is the frontier of a regular neighborhood of the union $E'_1 \cup \alpha_0 \cup E'_2$.
Then $E_1'$ and $E^*$ form a new complete meridian system of $W$.
See Figure \ref{fig:arcs_6}.
\begin{center}
\begin{overpic}[width=4cm,clip]{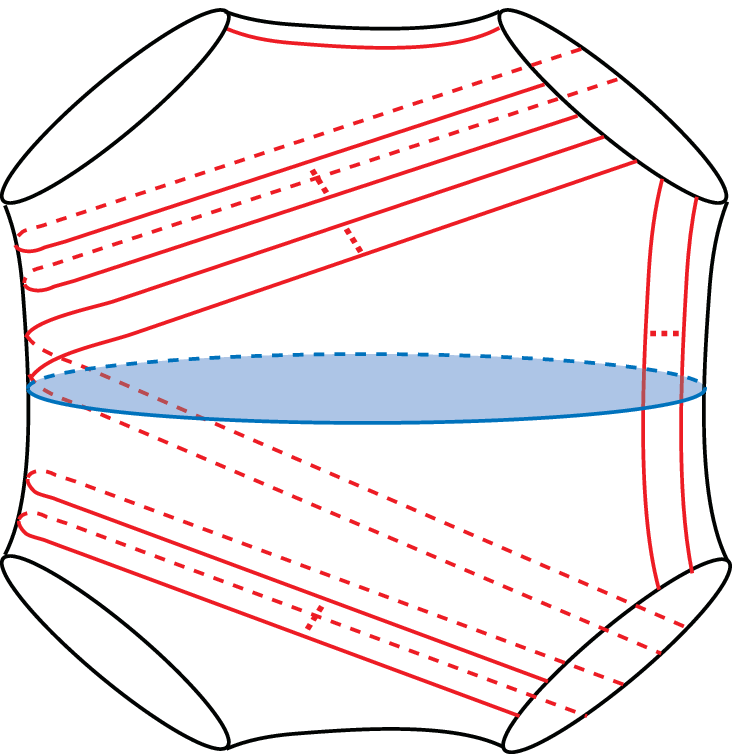}
\linethickness{3pt}
\put(0, 106){$e'^+_1$}
\put(0, 3){$e'^-_1$}
\put(100, 3){$e'^+_2$}
\put(100, 106){$e'^-_2$}
\put(42, 104){\color{red} \small $\alpha_0$}
\put(35, 14){\color{red} \small $\alpha_1$}
\put(115, 70){\color{red} \small $\alpha_2$}
\put(52, 53){\color{blue} \small $E^*$}
\end{overpic}
\captionof{figure}{The arcs $\Sigma' \cap \alpha$ and the disk $E^*$.}
\label{fig:arcs_6}
\end{center}
We assign the same symbol $y$ to $\partial E^*$.
Then the arc $\alpha$ determines
a word of the form
$x y^{\epsilon'_{1}} y^{\epsilon'_{2}} \cdots y^{\epsilon'_{l'}} x^{-1}$,
where $\epsilon'_{i} = \pm 1$ $(i=1, 2, \ldots, l)$,
$\sum_{i=1}^{l'} \epsilon'_{i} \neq 0$ and $l' < l$.
Applying this argument finitely many times, we end with the case
where $w$ is reduced (in particular $\epsilon_1 = \epsilon_l$), thus, the conclusion follows.
\end{proof}

\begin{lemma}
With a suitable choice of orientations of $\partial E_1'$ and $\partial E_2'$,
if a word in $\{x^{\pm 1}, y^{\pm 1} \}$ corresponding to $l$ contains
both of terms of the forms
$x y^{\epsilon_{1}} y^{\epsilon_{2}} \cdots y^{\epsilon_{l}}x$ and
$x y^{\delta_{1}} y^{\delta_{2}} \cdots y^{\delta_{k}}x$,
where $\epsilon_{i} = \pm 1$ for $i \in \{1 , 2, \ldots, l\}$,
$\delta_{j} = \pm 1$ for $j \in \{1, 2, \ldots, k\}$ and
$| \sum_{i=1}^{l} \epsilon_{i} - \sum_{j=1}^{k} \delta_{j} | \geq 2$,
then the element of $\pi_1 (W)$ represented by $l$ cannot be $($a positive power of $)$ a primitive element.
\label{lem:key3}
\end{lemma}
\begin{proof}
Set $w_1 = x y^{\epsilon_{1}} y^{\epsilon_{2}} \cdots y^{\epsilon_{l}}x$,
$w_2 = x y^{\delta_{1}} y^{\delta_{2}} \cdots y^{\delta_{k}}x$,
$m = \sum_{i=1}^{l} \epsilon_{i}$ and $n = \sum_{j=1}^{k} \delta_{j}$.
Let $\alpha$ and $\beta$ be the subarcs of $l$ corresponding to the subwords $w_1$ and $w_2$ respectively.
By cutting the Heegaard surface $\Sigma$ along $\partial E_1' \cup \partial E_2' $,
we get a 4-holed sphere $\Sigma'$.
Let $e'^\pm_{1}$ and $e'^\pm_{2}$ be the holes coming from $E_1'$ and $E_2'$ respectively.

Suppose first that both subwords $w_1$ and $w_2$ are reduced, so $w_1 = x y^m x$ and
$w_2 = x y^n x$.
If both of $m$ and $n$ are non-zero and these have different signs, then
we get the conclusion by Lemma \ref{lem:key}.
Thus, we assume that $m$ and $n$ have the same sign or one of them is zero.
Without loss of generality we can assume that $n > m \geq 0$.
Suppose that $m=0$. Then $w_1 = x^2$ and $w_2 = x y^n x$ ($n \geq 2$) and
thus, by Lemma \ref{lem:key}, $l$ cannot represent a primitive element of $\pi_1 (W)$.
Suppose that $m>0$.
Let $\alpha_0$ be the subarc of $\alpha$ corresponding to the term $xy$.
Then $\alpha_0$ connects two circles $e'^+_1$ and $e'^-_2$ in $\Sigma'$.
Let $E^*$ be the band sum of $E'_1$ and $E'_2$ along $\alpha_0$.
Then $E_1'$ and $E^*$ form a new complete meridian system of $W$.
See
Figure \ref{fig:arcs_4}.
\begin{center}
\begin{overpic}[width=3.5cm,clip]{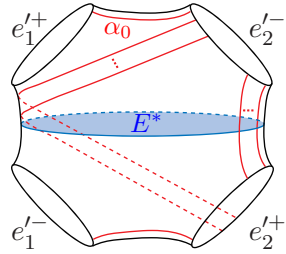}
\linethickness{3pt}
\put(0, 80){$e'^+_1$}
\put(0, 3){$e'^-_1$}
\put(90, 80){$e'^-_2$}
\put(90, 3){$e'^+_2$}
\put(35, 81){\color{red} \small $\alpha_0$}
\put(45, 44){\color{blue} \small $E^*$}
\end{overpic}
\captionof{figure}{The arcs $\Sigma' \cap (\alpha \cup \beta)$ and the disk $E^*$.}
\label{fig:arcs_4}
\end{center}
Assigning the same symbol $y$ to $\partial E^*$, the arc $\alpha$ determines
a word of the form $xy^{m-1}x$ while $\beta$ determines
$xy^{n-1}x$.
Applying this argument $m$ times, we finally end with the case of $m=0$, thus, the conclusion follows by induction.

Next suppose that at least one of $w_1$ or $w_2$ is reducible.
Without loss of generality we can assume that
$w_1$ is reducible and $\epsilon_{1} = 1$.
Since $w_1$ contains the term $y y^{-1}$, there is no arc in $\Sigma' \cap l$ connecting
$e'^-_{1}$ and $e'^+_{1}$.
See Figure \ref{fig:arcs_3}.
\begin{center}
\begin{overpic}[width=3.5cm,clip]{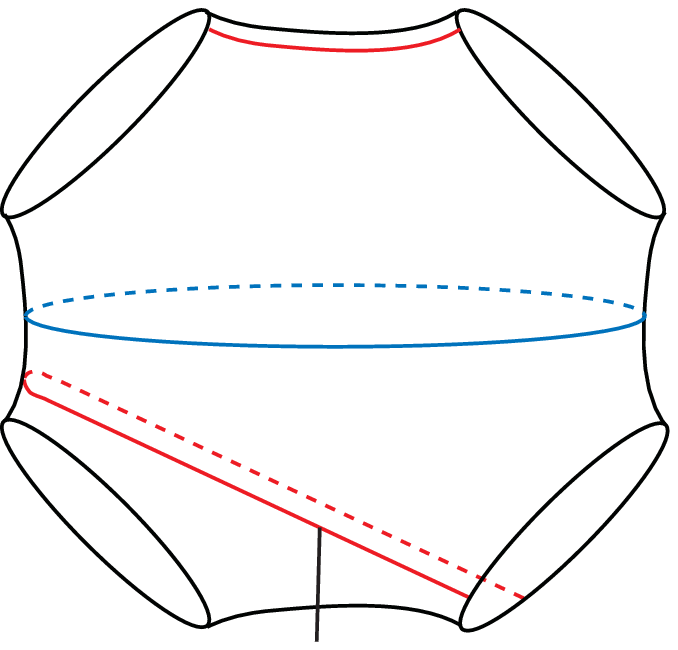}
\linethickness{3pt}
\put(0, 90){$e'^+_1$}
\put(0, 13){$e'^-_1$}
\put(90, 90){$e'^-_2$}
\put(90, 13){$e'^+_2$}
\put(42, 0){\small $y y^{-1}$}
\end{overpic}
\captionof{figure}{There are no arcs of $\Sigma' \cap l$ connecting
$e'^-_{1}$ and $e'^+_{1}$.}
\label{fig:arcs_3}
\end{center}
This implies that the word represented by $l$ cannot contain the term $x^2$, thus, $k \neq 0$.
Further if one of $\epsilon_{l}$, $\delta_1$ and $\delta_{k}$ is $-1$, then
$l$ cannot represent a primitive element of $\pi_1 (W)$
by Lemma \ref{lem:key}.
Thus, we can assume that $\epsilon_1 = \epsilon_l = \delta_1 = \delta_k = 1$.
Let $\alpha$ and $\beta$ be the subarcs of $l$ corresponding to the subwords $w_1$ and $w_2$ respectively.
Then $\Sigma' \cap ( \alpha \cup \beta )$ consits of arcs shown in Figure \ref{fig:arcs_5}.
Let $\alpha_0$ be the subarc of $\alpha$ correspoinding to the word $xy$ and
let $E^*$ be the band sum of $E'_1$ and $E'_2$ along $\alpha_0$.
\begin{center}
\begin{overpic}[width=4cm,clip]{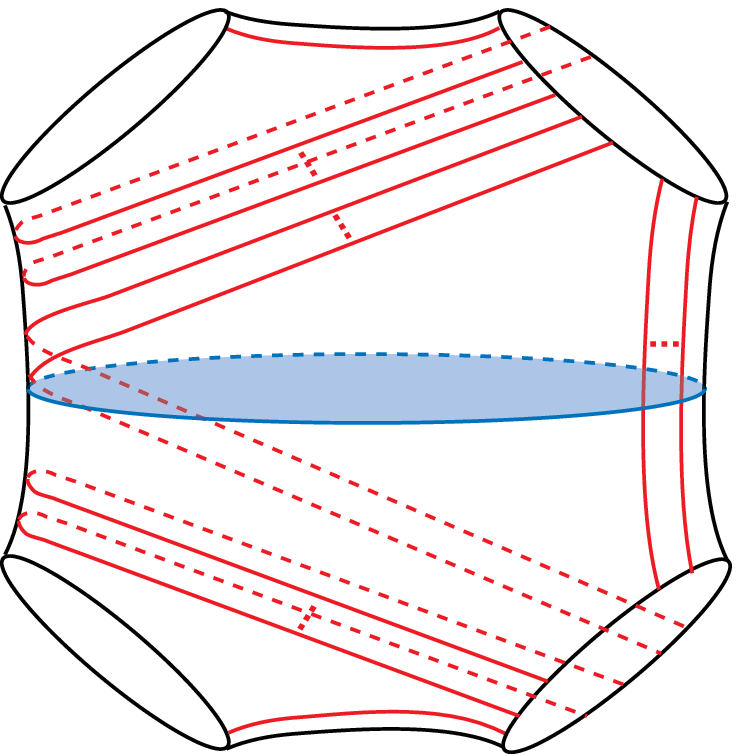}
\linethickness{3pt}
\put(0, 106){$e'^+_1$}
\put(0, 3){$e'^-_1$}
\put(100, 3){$e'^+_2$}
\put(100, 106){$e'^-_2$}
\put(42, 104){\color{red} \small $\alpha_0$}
\put(52, 53){\color{blue} \small $E^*$}
\end{overpic}
\captionof{figure}{The arcs $\Sigma' \cap (\alpha \cup \beta)$ and the disk $E^*$.}
\label{fig:arcs_5}
\end{center}
Then $E_1'$ and $E^*$ form a new complete meridian system of $W$.
Assigning the same symbol $y$ to $\partial E^*$, the arc $\alpha$ determines
a word of the form
$x y^{\epsilon'_{1}} y^{\epsilon'_{2}} \cdots y^{\epsilon'_{l'}}x$ ($\epsilon'_i = \pm 1$),
while $\beta$ determines
$x y^{\delta'_{1}} y^{\delta'_{2}} \cdots y^{\delta'_{k'}}x$ ($\delta'_j = \pm 1$).
Here note that we have $l' < l$, $k' < k$, $\sum_{i=1}^{l'} \epsilon'_{i}= m$ and
$\sum_{j=1}^{k'} \delta'_{j} = n$.
We repeat this argument until both words correspoinding to $\alpha$ and $\beta$ become reduced.
We claim that this process finishes in finitely many times.
Suppose not.
Then after repeating this process
finitely many times, we finally end with the case where
the word correspoinding to one of $\alpha$ and $\beta$ (say $\alpha$) is $x^2$.
Then again by Figure \ref{fig:arcs_3} the word corresponding to $\beta$ must be reduced.
This is a contradiction.
Now the conclusion follows from the argument of the case where both
$w_1$ and $w_2$ are reduced.
\end{proof}

\subsection{Shells}

Let $(V, W; \Sigma)$  be the genus-$2$ Heegaard splitting of a lens space $L = L(p, q)$ with $1 \leq q \leq p/2$.
We briefly review the definition of a {\it shell}, which is a special subcomplex of
the star neighborhood of a vertex of a primitive disk in the non-separating disk complex $\mathcal D(V)$,
introduced in Section 3.3 of \cite{CK15b}.
We call a pair of disjoint, non-isotopic primitive disks in $V$ a {\it primitive pair} in $V$.
A non-separating disk $E_0$ properly embedded in $V$ is said to be {\it semiprimitive} if there is a primitive disk $E'$ in $W$ disjoint from $E_0$.
A primitive pair and a semiprimitive disks in $W$ can be defined in the same way.

Let $E$ be a primitive disk in $V$.
Choose a dual disk $E'$ of $E$.
Then we have unique semiprimitive disks $E_0$ and $E'_0$ in $V$
and $W$ respectively that are disjoint from $E \cup E'$.
We denote the solid torus $\cl(V- \Nbd(E))$ simply by  $V_E$.
By a suitable choice of the oriented longitude and meridian of $V_E$, the circle $\partial E'_0$ can be assumed to be a $(p, \bar{q})$-curve on the boundary of $V_E$,
where $\bar{q} \in \{q, q'\}$ and $q'$ is the unique integer satisfying $1 \leq q' \leq p/2$ and $qq' \equiv \pm 1 \pmod p$.
We say that $E$ is of {\it $(p , \bar{q})$-type} if $\partial E'_0$ is a $(p, \bar{q})$-curve on $\partial V_E$.

Suppose first that $E$ is of $(p , q)$-type.
Then the circle $\partial E_0$ is a $(p, q')$-curve on the boundary of the solid torus $\cl(W- \Nbd(E'))$.
We construct a sequence of disks $E_0$, $E_1 , \ldots , E_p$ starting at the semiprimitive disk $\partial E_0$ as follows.
Choose an arc $\alpha_0$ on $\Sigma$ so that $\alpha_0$ meets each of
$E_0$ and $E$ in exactly one point of $\partial \alpha_0$, and $\alpha_0$ is disjoint from $\partial E' \cup \partial E'_0$.
Let $E_1$ be a disk in $V$ which is the band sum of $E_0$ and $E$ along $\alpha_0$.
Then the disk $E_1$ is disjoint from $E \cup E_0$ and intersects $\partial E'$ in a single point and $\partial E'_0$ in $p$ points.
Then inductively, we construct $E_{i+1}$ from $E_i$ until we get $E_p$ by taking the band sum with $E$ along an
arc $\alpha_i$ on $\Sigma$ such that $\alpha_i$ meets each of
$E_i$ and $E$ in exactly one point of $\partial \alpha_0$, and $\alpha_0$ is disjoint from $\partial E' \cup \partial E'_0$.
Figure \ref{fig:sequence_of_disks} illustrates the case when $E$ is of $(7, 3)$-type, and so $\partial E_0$ is a $(7, 2)$-curve on the boundary of the solid torus $\cl(W- \Nbd(E'))$ in the lens space $L(7, 3)$.

\begin{center}
\begin{overpic}[width=15cm,clip]{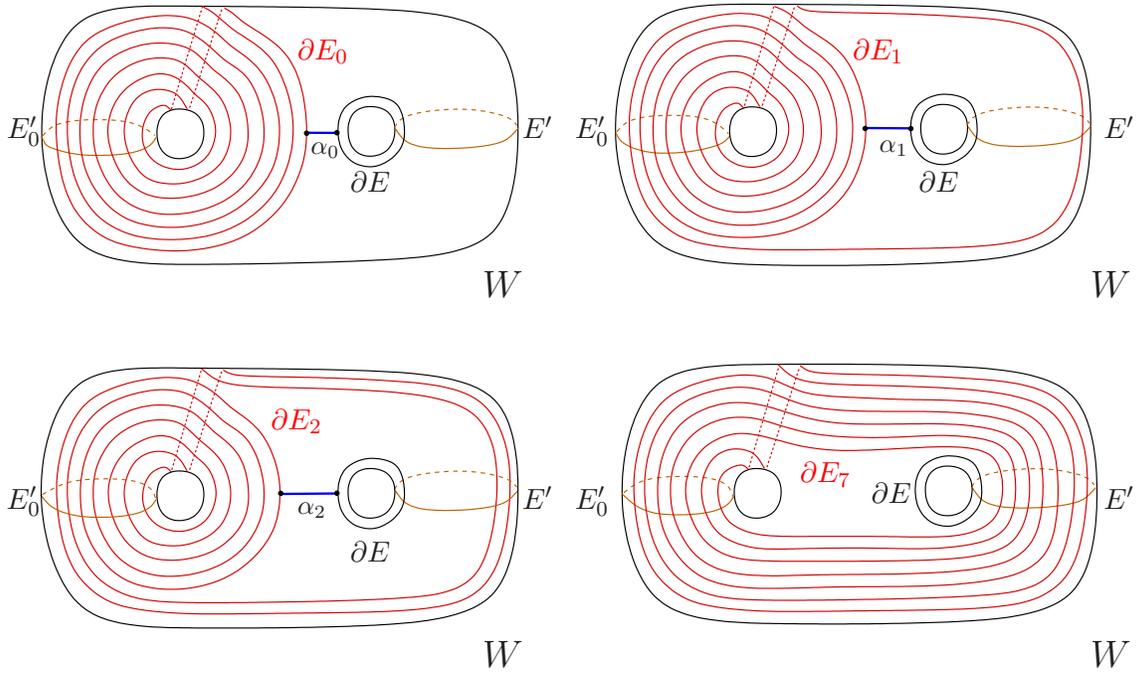}
 \linethickness{3pt}
  \put(0, 210){$E_0'$}
  \put(195, 210){$E'$}
  \put(215, 210){$E_0'$}
  \put(415, 210){$E'$}
  \put(0, 70){$E_0'$}
  \put(195, 70){$E'$}
  \put(215, 70){$E_0'$}
  \put(415, 70){$E'$}
  \put(115, 204){\small $\alpha_0$}
  \put(330, 205){\small $\alpha_1$}
  \put(110, 68){\small $\alpha_2$}
  \put(130, 190){$\partial E$}
  \put(345, 190){$\partial E$}
  \put(130, 50){$\partial E$}
  \put(327, 72){$\partial E$}
  \put(110, 240){{\color{red} $\partial E_0$}}
  \put(320, 240){{\color{red} $\partial E_1$}}
  \put(100, 100){{\color{red} $\partial E_2$}}
  \put(300, 80){{\color{red} $\partial E_7$}}
  \put(180, 150){\Large $W$}
  \put(410, 150){\Large $W$}
  \put(180, 10){\Large $W$}
  \put(410, 10){\Large $W$}
\end{overpic}
\captionof{figure}{The circles $\partial E_0$, $\partial E_1$, $\partial E_2$ and $\partial E_7$ on the surface $\Sigma = \partial W$ in $L(7, 2)$ for the sequence of disks $E_0, E_1, \ldots, E_7$.}
\label{fig:sequence_of_disks}
\end{center}

We call the full subcomplex of $\mathcal D(V)$ spanned by the vertices $E_0, E_1, \ldots, E_p$, and $E$ a {\it $(p, q)$-shell} centered at the primitive disk $E$
and denote it simply by $\mathcal S_E = \{E_0, E_1, \ldots, E_p\}$, see Figure \ref{fig:shell} for example.

\medskip

\begin{center}
\begin{overpic}[width=6cm,clip]{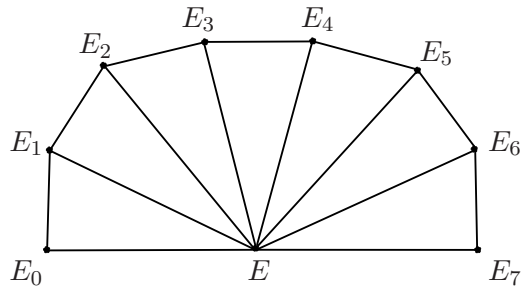}
 \linethickness{3pt}
  \put(80, 4){$E$}
  \put(-10, 4){$E_0$}
  \put(-10, 53){$E_1$}
  \put(16, 90){$E_2$}
  \put(55, 100){$E_3$}
  \put(100, 100){$E_4$}
  \put(145, 87){$E_5$}
  \put(171, 53){$E_6$}
  \put(171, 4){$E_7$}
\end{overpic}
\captionof{figure}{A $(7, 3)$-shell in the disk complex $\mathcal D(V)$ for $L(7, 3)$. The disks $E_1$, $E_2$, $E_5$, $E_6$ and $E$ are the only primitive disks by Lemma \ref{lem:sequence}.}
\label{fig:shell}
\end{center}

A shell is a 2-dimensional subcomplex of $\mathcal{D}(V)$.
It is straightforward from the construction that for $0 \leq i < j \leq p$, $E_i \cap E_j$ consists of $j - i -1$ arcs.
We note that there are infinitely many choice of such an arc $\alpha_0$, and so of the disk $E_1$. But once we choose $E_1$, the arcs $\alpha_i$ for $i \in \{1, 2, \ldots, p-1\}$ are uniquely determined and so are the successive disks $E_2$, $E_3, \ldots, E_p$.

Assign symbols $x$ and $y$ to the oriented circles $\partial E'$ and $\partial E'_0$ respectively.
Then the oriented boundary circles $\partial E$, $\partial E_0$, $\partial E_1, \ldots, \partial  E_p$ from the shell  $\mathcal S_E = \{E_0, E_1, \ldots, E_p\}$ represent elements of the free group $\pi_1 (W) = \langle x, y \rangle$.
We observe that, with a suitable choice of orientation,
the circles $\partial E$, $\partial E_0$, $\partial E_1, \partial E_2, \ldots, \partial E_{p-1}, \partial  E_p$
determine the words of the form $x$, $y^p$, $x y^{q}, xy^qxy^{p-q}, \ldots, (xy)^{p-1}y, (xy)^p$ respectively.

\begin{lemma}[Lemmas 3.8 and 3.13 in \cite{CK15b}]
\label{lem:sequence}
Let $\mathcal S_E = \{E_0, E_1, \ldots, E_{p-1}, E_p\}$ be a $(p, q)$-shell
centered at a primitive disk $E$ in $V$.
Then we have
\begin{enumerate}
\item $E_0$ and $E_p$ are semiprimitive.
\item $E_j$ is primitive if and only if $j \in \{1, q', p-q', p-1\}$ where $q'$ is the unique integer satisfying $qq' \equiv \pm1 \pmod p$ and $1 \leq q' \leq p/2$.
\item $E_1$ and $E_{p-1}$ are of $(p,q)$-type while $E_{q'}$ and $E_{p-q'}$ are of $(p,q')$-type.
\end{enumerate}
\end{lemma}

We have constructed a $(p, q)$-shell $\mathcal S_E$ by assuming $E$ is of a $(p, q)$-type.
If $E$ is of $(p, q')$-type, then $\mathcal S_E$ is a $(p, q')$-shell, and Lemma \ref{lem:sequence} still holds by exchanging $q$ and $q'$ in the conclusion.

\begin{lemma}[Lemma 3.6 in \cite{CK15b}]
\label{lem:first_surgery}
Let $E_0$ be a semiprimitive disk in $V$, and let $E$ be a primitive disk in $V$ disjoint from $E_0$.
If a primitive or semiprimitive disk $D$ in $V$  intersects $E \cup E_0$ transversely and minimally,
then the disk $E_1$ from surgery on $E \cup E_0$ along $D$ is a primitive disk, which has a common dual disk with $E$.
\end{lemma}

By Lemma \ref{lem:first_surgery}, given a primitive disk $E$ and a semiprimitive disk $E_0$ in $V$ disjoint from $E$, any primitive disk $D$ in $V$ determines a unique shell $\mathcal S_E = \{E_0, E_1, \ldots, E_{p-1}, E_p\}$ such that $E_1 = D$ if $D$ is disjoint from $E \cup E_0$, and $E_1$ is the disk from surgery on $E \cup E_0$ along $D$ if $D$ intersects $E \cup E_0$.
The following is a generalization of Lemma \ref{lem:first_surgery}.

\begin{lemma}[Lemma 3.10 in \cite{CK15b}]
\label{lem:surgery_on_primitive}
Let $\mathcal S_E = \{E_0, E_1, \ldots, E_{p-1}, E_p\}$ be a shell centered at a primitive disk $E$ in $V$, and let $D$ be a primitive or semiprimitive disk in $V$.
For $j \in \{1, 2, \ldots, p-1\}$,
\begin{enumerate}
\item if $D$ is disjoint from $E \cup E_j$ and is isotopic to none of $E$ and $E_j$, then $D$ is isotopic to either $E_{j-1}$ or $E_{j+1}$, and
\item if $D$ intersects $E \cup E_j$, then the disk from surgery on $E \cup E_j$ along $D$ is isotopic to either $E_{j-1}$ or $E_{j+1}$.
\end{enumerate}
\end{lemma}

\subsection{Bridges}
Let $(V, W; \Sigma)$  be the genus-$2$ Heegaard splitting of a lens space $L = L(p, q)$ with $1 \leq q \leq p/2$ and $p \not\equiv \pm 1 \pmod q$.
Throughout the subsection, we fix the followings.
\begin{itemize}
\item
An integer $\bar{q} \in \{q, q' \}$ where $q'$ is the unique integer satisfying $1 \leq q' \leq p/2$ and $qq' \equiv \pm 1 \pmod p$; and
\item
The integers $m$ and $r$ satisfying $p=\bar{q} m + r$ with $2 \leq r \leq q-2$.
\end{itemize}

We recall from Theorem \ref{thm:contractibility} that $\mathcal P(V)$ consists of infinitely many isomorphic tree components in this case.
The key of the disconnectivity is that we can find non-adjacent primitive disks
$D$ and $E$ in $V$ such that
the corridor connecting $D$ and $E$ contains no vertices of primitive disks except $D$ and $E$.
Then $D$ and $E$ are contained in different components of $\mathcal{P}(V)$
since the dual complex of $\mathcal{D}(V)$ is a tree.

We call a corridor $\mathcal{C}_{\{ D, E \} } = \{  \Delta_1 , \Delta_2 , \ldots, \Delta_n \}$ in $\mathcal{D}(V)$ a {\it bridge} when it connects vertices $D$ and $E$ of primitive disks, and contains no vertices of primitive disks except $D$ and $E$.
In this case, we denote the bridge by $\mathcal B_{\{D, E\}}$ instead of $\mathcal{C}_{\{ D, E \} }$.
We recall that the two vertices of the $2$-simplex $\Delta_1$ other than $E$ forms the principal pair $\{E_*, E_{**}\}$ of $E$ with respect to $D$ (Lemma \ref{lem:principal pair of a corridor}).

\begin{lemma}[Lemma 4.1 and Theorem 4.2 in \cite{CK15b}]
\label{lem:construction of a bridge}
Let $E$ be a primitive disk of $(p,\bar{q})$-type.
Let $\mathcal S_E = \{E_0, E_1, \ldots, E_p\}$ be a $(p, \bar{q})$-shell centered at $E$.
Then there exists a primitive disk $D$ and a bridge $\mathcal{B}_{\{ D , E \}} = \{ \Delta_1 , \Delta_2 , \ldots , \Delta_n \}$
connecting $E$ and $D$ such that $\Delta_1 = \{ E, E_m , E_{m+1} \}$.
\end{lemma}
\begin{proof}
Actually, the proof of Theorem 4.2 in \cite{CK15b} introduces an algorithmic way to produce a bridge under the condition $p \not\equiv \pm 1 \pmod q$.
We provide a sketch of the proof.
Assigning symbols $x$ and $y$ to $\partial E'$ and $\partial E'_0$ respectively, with
appropriate orientations,
any oriented simple closed curve on $\Sigma$ represents an element of the free group $\pi_1 (W) = \langle x, y \rangle$.
Then we may assume that $\partial E_m$ and $\partial E_{m+1}$ represents the elements
$(xy^{\bar{q}})^{m-1} x y^{\bar{q}+r}$ and $(xy^{\bar{q}})^{m} x y^{r}$ respectively.
By cutting $\Sigma = \partial V$ along $\partial E_m \cup \partial E_{m+1}$, we get a $4$-holed sphere $\Sigma_*$ shown in Figure \ref{fig:L_or_R-replacement}.
In the figure, $D_1 = E_m$ and $D_2 = E_{m+1}$, and $m_1 = m -1$ and $m_2 = m$.
We have two possibilities of the patterns of $\partial E' \cap \Sigma_*$ as in the figure.

\begin{center}
\begin{overpic}[width=15cm,clip]{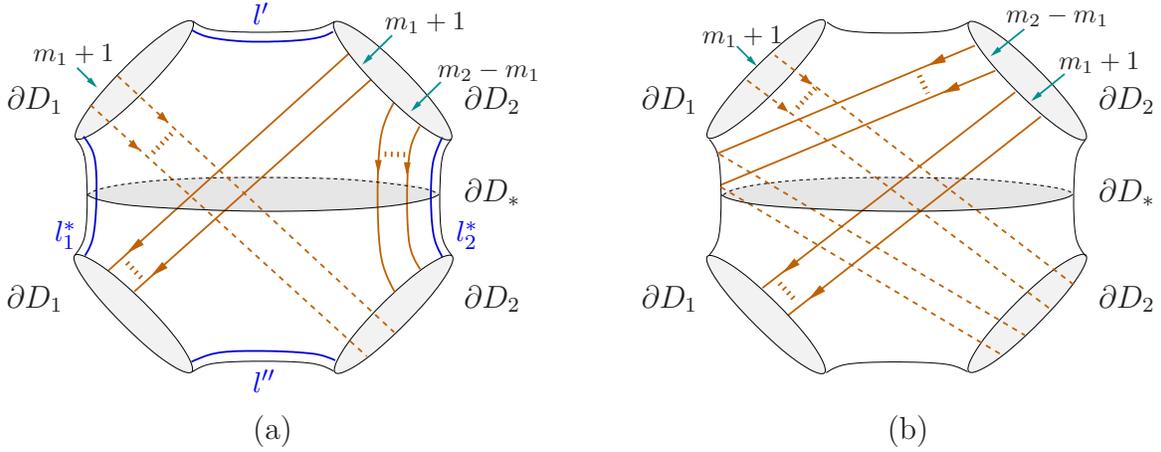}
 \linethickness{3pt}
  \put(7, 143){\small $m_1 + 1$}
  \put(140, 155){\small $m_1 + 1$}
  \put(160, 137){\small $m_2 - m_1$}
  \put(260, 150){\small $m_1 + 1$}
  \put(395, 140){\small $m_1 + 1$}
  \put(375, 158){\small $m_2 - m_1$}
  \put(90, 157) {${\color{blue}l'}$}
  \put(90, 18) {${\color{blue}l''}$}
  \put(15, 75) {${\color{blue}l^*_1}$}
  \put(167, 75) {${\color{blue}l^*_2}$}
  \put(-3, 125){\large$\partial D_1$}
  \put(170, 125){\large $\partial D_2$}
  \put(237, 125){\large $\partial D_1$}
  \put(410, 125){\large $\partial D_2$}
  \put(-3, 50){\large $\partial D_1$}
  \put(170, 50){\large $\partial D_2$}
  \put(237, 50){\large $\partial D_1$}
  \put(410, 50){\large $\partial D_2$}
  \put(170, 90){\large $\partial D_*$}
  \put(410, 90){\large $\partial D_*$}
  \put(90, 0){\large (a)}
  \put(330, 0){\large (b)}
\end{overpic}
\captionof{figure}{The $4$-holed sphere $\Sigma_*$. There are two patterns of $\partial E' \cap \Sigma_*$.}
\label{fig:L_or_R-replacement}
\end{center}

For each of the two cases, we can show that the boundary circle of the horizontal disk $E_\emptyset$, which is $\partial D_*$ in the figure,
represents the word $(xy^{\bar{q}})^{2m} x y^{2r}$.
We call $\{ E_m , E_{\emptyset} \}$ ($\{ E_{\emptyset} , E_{m+1} \}$, respectively)
the pair obtained from the pair $\{ E_m ,  E_{m+1} \}$ by
{\it $R$-replacement} ({\it $L$-replacement}, respectively).
By cutting $\partial V$ along $ \partial E_m \cup \partial E_{\emptyset} $ ($ \partial E_{\emptyset} \cup \partial E_{m+1}$, respectively)
we get again a $4$-holed sphere, denoted by $\Sigma_*$ again, and the two possibilities of the patterns of $\partial E' \cap \Sigma_*$
shown in Figure \ref{fig:L_or_R-replacement}.
We denote by $E_R$ ($E_L$, respectively) the horizontal disk in the figure.
The boundary of $E_R$ ($E_L$, respectively) represents the word $(xy^{\bar{q}})^{3m} x y^{3r}$
($(xy^{\bar{q}})^{3m+1} x y^{3r - \bar{q}}$, respectively).
In the figure, $D_* = E_R$, $D_1 = E_m$, $D_2 = E_{\emptyset}$, $m_1 = m -1$ and $m_2 = 2m$.
($D_* = E_L$, $D_1 = E_{\emptyset}$, $D_2 = E_{m+1}$, $m_1 = 2m$ and $m_2 = m$, respectively.)
We call $\{ E_m , E_R \}$ and $\{ E_R , E_{\emptyset} \}$ ($\{ E_{\emptyset} , E_L \}$ and $\{ E_L , E_{m+1} \}$, respectively)
the pair obtained from the pair $\{ E_m , E_{\emptyset} \}$ ($\{ E_{\emptyset} , E_{m+1} \}$, respectively) by
$R$-replacement and $L$-replacement respectively.
In this way, for each positive word $w = w(L, R)$ on the set of letters $\{ L, R \}$,
we can define the disk $E_w$ inductively.
We call the full subcomplex $\mathcal{X}$ of the disk complex $\mathcal{D}(V)$
spanned by the set of vertices
\[\{ E, E_m , E_{m+1} \} \cup \{ E_w \mid \mbox{$w$ is a (possibly empty) positive word
in $\{ L, R \}$}  \}\]
the {\it principal complex generated by} $\mathcal{S}_E$.
See Figure \ref{fig:principal_complex}.

\begin{center}
\begin{overpic}[width=11cm,clip]{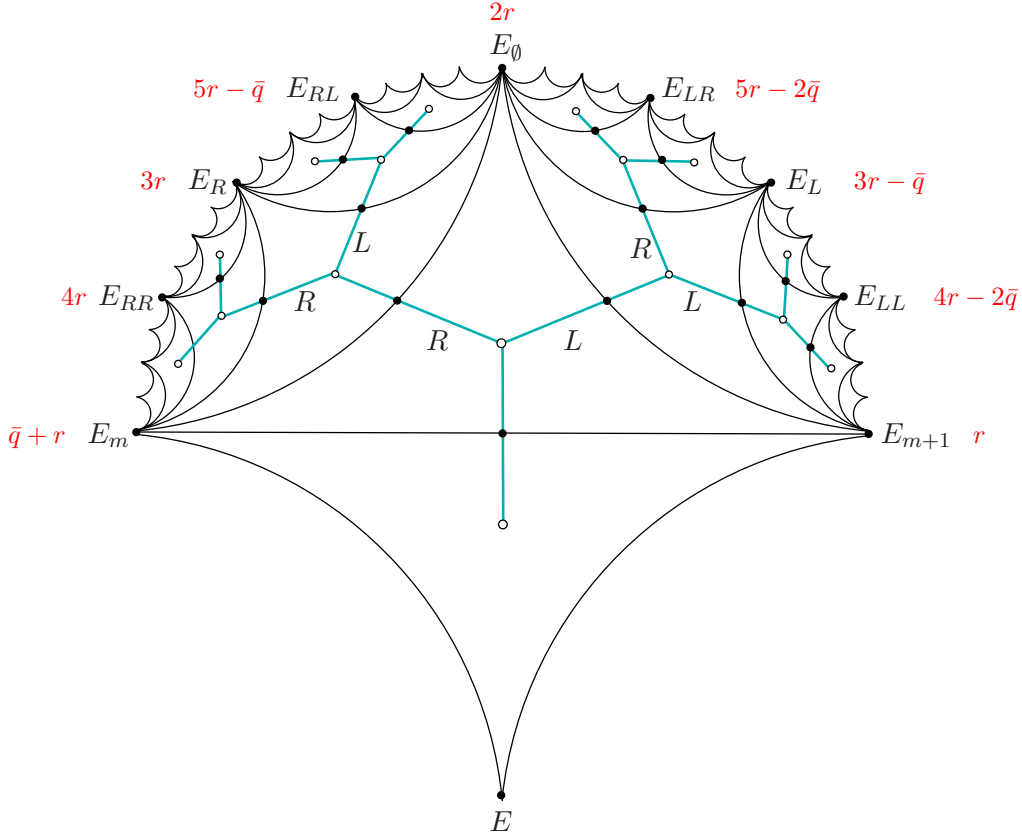}
 \linethickness{3pt}
  \put(152, 8){$E$}
  \put(0, 155){$E_m$}
  \put(300, 155){$E_{m+1}$}
  \put(4, 207){$E_{RR}$}
  \put(290, 207){$E_{LL}$}
  \put(38, 251){$E_R$}
  \put(264, 251){$E_L$}
  \put(75, 286){$E_{RL}$}
  \put(217, 286){$E_{LR}$}
  \put(152, 302){$E_\emptyset$}
  \put(128, 190){$R$}
  \put(78, 205){$R$}
  \put(180, 190){$L$}
  \put(225, 205){$L$}
  \put(100, 227){$L$}
  \put(205, 225){$R$}

  \put(-30, 155){\color{red} \small $\bar{q}+r$}
  \put(335, 155){\color{red} \small  $r$}
  \put(-10, 207){\color{red} \small  $4r$}
  \put(320, 207){\color{red} \small  $4r-2\bar{q}$}
  \put(20, 251){\color{red} \small  $3r$}
  \put(290, 251){\color{red} \small  $3r-\bar{q}$}
  \put(40, 286){\color{red} \small  $5r-\bar{q}$}
  \put(245, 286){\color{red} \small  $5r-2 \bar{q}$}
  \put(152, 315){\color{red} \small  $2r$}

\end{overpic}
\captionof{figure}{The principal complex generated by $\mathcal{S}_E$.}
\label{fig:principal_complex}
\end{center}

The words in
$\{ x^{\pm 1}, y^{\pm 1} \}$ represented by vertices of
the principal complex $\mathcal{X}$ has the following property:
\begin{itemize}
\item
Let $\{D_1, D_2\}$ be a pair of vertices in $\mathcal X$ obtained from $\{ E_m, E_{m+1} \}$ by a series of replacements, and let $D_*$ be the vertex such that $\{D_1, D_*\}$ and $\{D_*, D_2\}$ are the $R$-replacement and $L$-replacement of $\{D_1, D_2\}$ respectively.
If $\partial D_1$ and $\partial D_2$ represent the words $(xy^{\bar{q}})^{m_1} x y^{n_1}$ and
$(xy^{\bar{q}})^{m_2} x y^{n_2}$ respectively,
then the word represented by
$\partial D_*$ is
$(xy^{\bar{q}})^{m_1 + m_2 + 1} x y^{n_1 + n_2 - \bar{q}}$.
See Figure \ref{fig:L_or_R-replacement}.
\end{itemize}
Using this property, we can find inductively the word
corresponding to each vertex in $\mathcal{X}$.
Further, we see that  every word corresponding to
a vertex $E_w$ in $\mathcal{X}$ is a positive power of
$x y^{\bar{q}}$ followed by $x y^{n_w}$ for some non-zero integer
$n_w$.
In Figure \ref{fig:principal_complex}, this number $n_w$ is assigned for each vertex.

By an elementary argument on number theory,
we can show that there exists a vertex $E_w$ in $\mathcal{X}$
such that $n_w =  \bar{q} \pm 1$.
This implies that $E_w$ is a primitive disk.
Choose such $w$ so that
the length of $w$ is minimal, and set $D = E_w$.
Then the corridor connecting $E$ and $D$
is the required bridge.
(Here we recall that neither of $E_m$ and $E_{m+1}$ is primitive.)
\end{proof}

\begin{lemma}
\label{lem:principal pair of a bridge}
Let $\mathcal{B}_{ \{ D, E \} }$ be a bridge connecting vertices $D$ and $E$ of primitive disks.
Let $\{ E_{*} , E_{**} \}$ be the principal pair of $E$ with respect to $D$.
Let $E$ be of $(p , \bar{q})$-type.
Then there exists a unique
$(p, \bar{q})$-shell $\mathcal S_E = \{E_0, E_1, \ldots, E_p\}$ centered at
$E$ such that $\{ E_* , E_{**} \} = \{ E_m , E_{m+1} \}$.
\end{lemma}

\begin{proof}
The proof will be similar to that of Theorem 3.11 in \cite{CK15b}.
Let $C$ be an outermost subdisk of $D$ cut off by $D \cap E$.
Any dual disk $E'$ of $E$ determines a unique semiprimitive disk $E_0$ in $V$ disjoint from $E \cup E'$.
Among all the dual disks of $E$, choose one, denoted by $E'$ again, so that the resulting semiprimitive disk $E_0$ intersects $C$ minimally.
Then, $C$ must intersect $E_0$, otherwise one of the disks of the principal pair $\{E_*, E_{**}\}$ is $E_0$ and the other is primitive by Lemma \ref{lem:first_surgery}, which is impossible since they are the representatives of the vertices of the bridge $\mathcal{B}_{ \{ D, E \} }$ other than $D$ and $E$.

Let $C_0$ be an outermost subdisk of $C$ cut off by $C \cap E_0$.
Then one of the disks from surgery on $E_0$ along $C_0$ is $E$, and the other, say $E_1$, is primitive by Lemma \ref{lem:first_surgery} again.
Then we have the unique shell $\mathcal S_E = \{E_0, E_1, E_2, \ldots, E_p\}$ centered at $E$.
Let $E'_0$ be  a semiprimitive disk in $W$ disjoint from $E \cup E'$.
The circle $\partial E'_0$ would be a $(p, \bar q)$-curve on the boundary of the solid torus $\cl(V- \Nbd(E \cup E'))$.
We will assume $\bar q = q$.
That is, $\partial E'_0$ is a $(p, q')$-curve and so $\mathcal S_E$ is a $(p, q)$-shell.
The proof is easily adapted for the case of $\bar q = q'$.

\medskip

{\noindent \sc Claim 1.} There is a disk $E_j$ in the shell for some $1 \leq j < p/2$ that is disjoint from $C$.

\smallskip

{\noindent \it Proof of Claim $1$.}
First, it is clear that there is a disk $E_j$ for some $j \in \{1, 2, \ldots, p-1\}$ disjoint from $C$.
Since if $C$ intersects each of $E_1, E_2, \ldots, E_j$, for $j \in \{1, 2, \ldots, p-1\}$, then one of the disks from surgery on $E_j$ by an outermost subdisk $C_j$ of $C$ cut off by $C \cap E_j$ is $E$ and the other one is $E_{j+1}$ by Lemma \ref{lem:surgery_on_primitive}, and we have $|C \cap E_{j+1}| < |C \cap E_j|$.
Consequently, we see that $|C \cap E_p| < |C \cap E_0|$, but it contradicts the minimality of $|C \cap E_0|$ since $E_p$ is also a semiprimitive disk disjoint from $E$.

Now, denote by $E_j$ again the first disk in the sequence that is disjoint from $C$.
Then the two disks from surgery on $E$ along $C$ are $E_j$ and $E_{j+1}$, hence, $C$ is also disjoint from $E_{j+1}$.
Actually they are the only disks in the sequence disjoint from $C$.
For other disks in the sequence, it is easy to see that $|C \cap E_{j-k}| = k = |C \cap E_{j+1+k}|$.
If $j \geq p/2$, then we have $|C \cap E_0| = j > p-j-1 = |C \cap E_p|$, a contradiction for the minimality condition again.
Thus, $E_j$ is one of the disks in the first half of the sequence, that is, $1 \leq j < p/2$.

\medskip

{\noindent \sc Claim 2.}
The disk $E_j$ is $E_m$.

\smallskip

{\noindent \it Proof of Claim $2$.}
First, it is clear that $E_j$ is not $E_1$. That is, $C$ must intersect $E_1$, otherwise the principal pair $\{E_*, E_{**}\}$ equals $\{E_1, E_2\}$ and $E_1$ is primitive, which is impossible.
Assigning symbols $x$ and $y$ to oriented $\partial E'$ and $\partial E_0'$ respectively, $\partial E_1$, $\partial E_2$, $\partial E_3$ may represent the elements of the form $xy^p$, $xy^qxy^{p-q}$, $xy^qxy^qxy^{p-2q}$ respectively.
In general, $\partial E_k$ represents an element of the form $xy^{n_1}xy^{n_2} \cdots xy^{n_k}$ for some positive integers $n_1, \ldots, n_k$ with $n_1+ \cdots + n_k = p$ for each $k \in \{1, 2, \ldots, p\}$.
Furthermore, since $C$ is disjoint from $E_j$ and also from $E_{j+1}$, the word determined by the circle $\partial D$ contains the subword of the form $y^{m_1}xy^{m_2} \cdots xy^{m_{j+1}}$ (or its reverse) which is the part of $\partial C$ when $\partial E_{j+1}$ represents an element of the form $xy^{m_1}xy^{m_2} \cdots xy^{m_{j+1}}$.

Suppose that $2 \leq j \leq m-1$. Then an element represented by $\partial E_{j+1}$ has the form $xy^q \cdots xy^q xy^{p-jq}$, and so an element represented by $\partial D$ contains $xy^qx$ and $y^{p-jq}$, which lies in the part of $\partial C$.
Since $(p-jq)-q = q(m-1-j)+r \geq 2$, $D$ cannot be primitive by Lemma \ref{lem:property_of_primitive_elements2}, a contradiction.

Suppose that $m+1 \leq j \leq q'-2$.
We may write a word of $\partial E_{q'}$ as $xy^{n_1}xy^{n_2} \cdots xy^{n_{q'}}$ where each $n_k \in \{n, n+1\}$ for some positive integer $n$ since $\partial E_{q'}$ is primitive.
Then, by a similar consideration to the above, an element represented by $\partial D$ contains $xy^{m_1}x$ and $y^{m_2}$ for some positive integers $m_1$ and $m_2$ with $|m_1 - m_2|\geq2$, which lies in the part of $\partial C$.
Thus, $D$ cannot be primitive by Lemma \ref{lem:key3}, a contradiction again.

Suppose that $q'-1 \leq j \leq q'$.
Then the principal pair $\{E_*, E_{**}\}$ contains $E_{q'}$ which is primitive by Lemma \ref{lem:sequence}. This is impossible.

Finally, suppose that $q'+1 \leq j < p/2$.
Then, considering an element represented by $\partial E_{j+1}$, we observe that an element represented by $\partial D$ contains $xyx$ which lies in the part of $\partial C$.
Furthermore, when we write a word of the part of $\partial C$ as $xy^{n_1}xy^{n_2} \cdots xy^{n_{j+1}}$, at least two of $n_1, n_2, \ldots, n_{j+1}$ are $1$ and so at least one of $n_1, n_2, \ldots, n_{j+1}$ is greater than $2$.
Again, $D$ cannot be primitive by Lemma \ref{lem:property_of_primitive_elements2}, a contradiction.

\smallskip

From {\sc Claim 2}, the outermost subdisk $C$ is disjoint from $E_m$, and hence, the principal pair $\{E_*, E_{**}\}$ equals $\{E_m, E_{m+1}\}$.
\end{proof}

\begin{lemma}
Let $E$ be a primitive disk of $(p , \bar{q})$-type, and let $\mathcal S_E = \{E_0, E_1, \ldots, E_p\}$
be a $(p, \bar{q})$-shell centered at $E$.
Let $D$ be a primitive disk and let $\mathcal{B}_{ \{ D, E \} } = \{  \Delta_1 , \Delta_2 , \ldots, \Delta_n \} $ be the bridge with $\Delta_1 = \{E, E_m, E_{m+1}\}$, given by Lemma $\ref{lem:construction of a bridge}$.
Given a primitive disk $F$ and a corridor $\mathcal{C}_{ \{ F, E \} } = \{  \Delta'_1 , \Delta'_2 , \ldots, \Delta'_{n'} \} $, if $\Delta'_1 = \Delta_1$, then the corridor $\mathcal C_{\{F, E\}}$ contains the bridge $\mathcal B_{\{D, E\}}$.
\label{lem:words for disks in a bridge}
\end{lemma}

\begin{proof}
We assume that each of the disks $\partial D$, $\partial E$, $\partial F$ intersects
$\partial E' \cup \partial E'_0$ transversely and minimally.
Assigning symbols $x$ and $y$ to $\partial E'$ and $\partial E'_0$ respectively,
with appropriate orientations, a word for $\partial E$ is $x$ while
a word for $\partial D$ is a positive power of
$x y^{\bar{q}}$ followed by $x y^{\bar{q}\pm 1}$
as in the proof of Lemma \ref{lem:construction of a bridge}.

For each  $k \in \{1, 2, \ldots, n-1\}$, set $\Delta_{k} \cap \Delta_{k+1} = \{ D_k, D'_k \}$ and $\Delta_{k+1} = \{ D_{k+1}, D_k, D'_k \}$.
To show the corridor $\mathcal C_{\{F, E\}}$ contains the bridge $\mathcal B_{\{D, E\}}$, it suffices to show that the disk from surgery on $D_k \cup D'_k$ along $F$ is  $D_{k+1}$ by Lemma \ref{lem:two corridors}.
We use the induction on $k$.
If $k = 1$, the conclusion holds immediately since we assumed $\Delta'_1 = \Delta_1$.

Let $k \geq 2$ and assume the conclusion holds for all $i < k$.
We simply write $\Delta_{k} = \{ D_0, D_1, D_2 \}$ and $\Delta_{k+1} = \{ D_1, D_2, D_* \}$ from now on, and will show that the disk from surgery on $D_1 \cup D_2$ along $F$ is  $D_*$.
By the construction of a bridge in the proof of Lemma \ref{lem:construction of a bridge}, we may assume that the circles $\partial D_1$ and $\partial D_2$ represent the elements of the forms $(xy^{\bar{q}})^{m_1} xy^{n_1}$
and $(xy^{\bar{q}})^{m_2} xy^{n_2}$ respectively,
where $n_1 > \bar{q} > n_2 > 0$.
By cutting the Heegaard surface $\Sigma$ along $\partial D_1 \cup \partial D_2$,
we get a $4$-holed sphere $\Sigma_*$.
We denote by $d^\pm_{1}$ and $d^\pm_{2}$ the holes coming from $\partial D_1$ and $\partial D_2$
respectively.
There are two patterns of $\partial E' \cap \Sigma_*$ as in Figure \ref{fig:L_or_R-replacement}, and the boundary of the disk $D_*$ is the horizontal circle in the figure.
We only consider the case of Figure \ref{fig:L_or_R-replacement} (a). The argument for the case (b) will be the same.

Let $C$ be an outermost subdisk of $F$ cut off by $D_1 \cup D_2$.
Then $\alpha = \Sigma_* \cap C$ is an arc whose end points lie in the same component of $\partial \Sigma_*$.
We consider only the case where $\partial \alpha$ lies in $d^-_{1}$.
The argument for other cases will be the same.
Let $\widetilde{\Sigma}_*$ be the covering space of $\Sigma_*$ such that
\begin{enumerate}
\item
$\widetilde{\Sigma}_*$ is the plane $\mathbb R^2$ with an open disk of radius at most $1/8$ removed
from each point with integer coordinates;
\item
the components of the preimage of $l^*_1$ ($l^*_2$, respectively)
are the vertical lines with even (odd, respectively) integer $x$-coordinate; and
\item
the components of the preimage of $l'$ ($l''$, respectively)
are the horizontal lines with even (odd, respectively) integer $y$-coordinate.
\end{enumerate}
We put a lift of $d^-_{1}$ at the origin.
Denote by $\Rational_{\mathrm{odd}}$ the set of irreducible rational numbers with odd denominators.
Then the set $\Rational_{\mathrm{odd}}$ one-to-one corresponds to the set of the (isotopy classes of) essential arcs $\alpha = \alpha_{r/s}$ on $\Sigma_*$ such that both endpoints of $\alpha$ lie in $d^-_{1}$ and $\alpha$ cuts off an annulus one of whose boundary circle is either $d^+_2$ or $d^-_2$, as follows.
For each $r/s \in \Rational_{\mathrm{odd}}$,
let $\widetilde{\beta} = \widetilde{\beta}_{r/s}$ be the line segment of the slope $r/s$ connecting the origin and the point $(s, r)$ in $\mathbb R^2$ such that $\widetilde{\beta} \cap \widetilde{\Sigma}_*$ is a single arc properly embedded in $\widetilde{\Sigma}_*$ (by assuming the open disks removed are sufficiently small).
Then the image $\beta$ of the segments $\widetilde{\beta} \cap \widetilde{\Sigma}_*$ is a simple arc in $\Sigma_*$ connecting $d_{1}^-$ and one of $d^+_2$ or $d^-_2$, say $d^+_2$.
Then the corresponding essential arc $\alpha = \alpha_{r/s}$ is the frontier of a regular neighborhood of the union
$\beta \cup d^+_2$.

Recall that our arc $\alpha \subset \Sigma_*$ is the intersection of
$\Sigma_*$ and an outermost subdisk $C$ of $F$ cut off by $D_1 \cup D_2$.
Let $r/s$ ($s>0$) be the element of $\Rational_{\mathrm{odd}}$ corresponding to $\alpha$.
We shall prove that $r/s = 0$, which implies that the disk from surgery on $D_1 \cup D_2$ along $F$ is $D_*$.
To show this, we will check all the cases of $r \neq 0$ violate the assumption that $F$ is primitive.

Figure \ref{fig:parameters} illustrates the patterns of $\Sigma_* \cap (\partial E' \cup \partial E'_0)$ on the 4-holed sphere $\Sigma_*$.
\begin{center}
\begin{overpic}[width=5cm,clip]{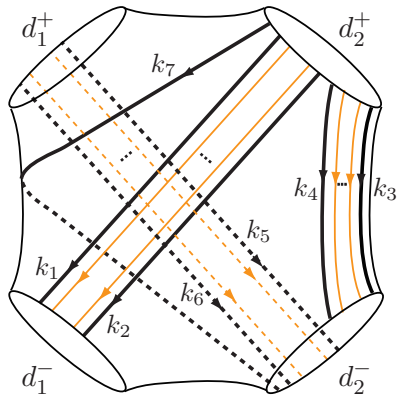}
 \linethickness{3pt}
\put(5, 135){$d^+_1$}
\put(5, 3){$d^-_1$}
\put(125, 135){$d^+_2$}
\put(125, 3){$d^-_2$}
\put(10, 45){\small $k_1$}
\put(37, 23){\small $k_2$}
\put(138, 75){\small $k_3$}
\put(108, 75){\small $k_4$}
\put(90, 59){\small $k_5$}
\put(65, 35){\small $k_6$}
\put(54, 122){\small $k_7$}
\end{overpic}
\captionof{figure}{The patterns of  $\Sigma_* \cap (\partial E' \cup \partial E'_0)$.}
\label{fig:parameters}
\end{center}
In the figure, each oriented bold arc, we denote by $\gamma'$, with a weight $k_i$ ($i \in \{ 1, 2, \ldots, 7  \} $) indicates the union of some arcs
$\gamma'_1$, $\gamma'_2 , \ldots, \gamma'_l$ of
$\Sigma_* \cap \partial E'_0$ on $\Sigma_*$ parallel to $\gamma'$
such that $k_i  = \sum_{i=j}^l [\gamma'_j]$ in $H_1(\Sigma_* , \partial \Sigma_*; \Integer)$,
after equipping the orientation of $\gamma'_j$ compatible with that of $\partial E'_0$.
Between the arcs with weights $k_1$ and $k_2$ (and also between the arcs with weights $k_5$ and $k_6$), we have $m_1 + 1$ oriented arcs $\gamma_1$, $\gamma_2, \ldots, \gamma_{m_1 +1}$ of $\Sigma_* \cap \partial E'$, and between $\gamma_i$ and $\gamma_{i+1}$ for $i \in \{1, 2, \ldots, m_1\}$, we have exactly $\bar{q}$ parallel arcs of
$\Sigma_* \cap \partial E_0'$ on $\Sigma_*$ in the same direction.
Similarly, between the arcs with weights $k_3$ and $k_4$, we have $m_2 - m_1$ oriented arcs of $\Sigma_* \cap \partial E'$, and between any two consecutive arcs of them, we have exactly $\bar{q}$ parallel arcs of $\Sigma_* \cap \partial E_0'$ on $\Sigma_*$ in the same direction.
By Figure \ref{fig:L_or_R-replacement} we have the following linear equations:
\begin{eqnarray}
\label{eq:solution of the linear equation}
\left\{
\begin{array}{ll}
k_1 + k_2 = n_1\\
k_5 + k_6 = n_1\\
k_1 + k_3 + k_7 = n_2\\
k_4 + k_6 + k_7 = n_2\\
k_2 + k_4 = \bar{q}\\
k_3 + k_5 = \bar{q}
\end{array}
\right.
\end{eqnarray}
Solving these equation, we see that there exist non-negative integers $a$ and $b$ such that
$k_1 = a$, $k_2 = n_1 - a$, $k_3 = n_2 - a - b$, $k_4 = \bar{q} - n_1 + a$, $k_3 = \bar{q} - n_2 + a + b$, $k_6 = n_1 + n_2 - \bar{q} - a - b$, $k_7 = b$.

\medskip

\noindent \textit{Case }1. $r > s$.

In this case with a suitable choice of an orientation of $\alpha$ the word corresponding to the arc $\alpha$ contains both of the terms
$x y^{\bar{q}} x$ and $x y^{n_2} x$ after canceling pairs of $y$ and $y^{-1}$ if necessary.
See Figure \ref{fig:subarcs_1}.
\begin{center}
\begin{overpic}[width=12cm,clip]{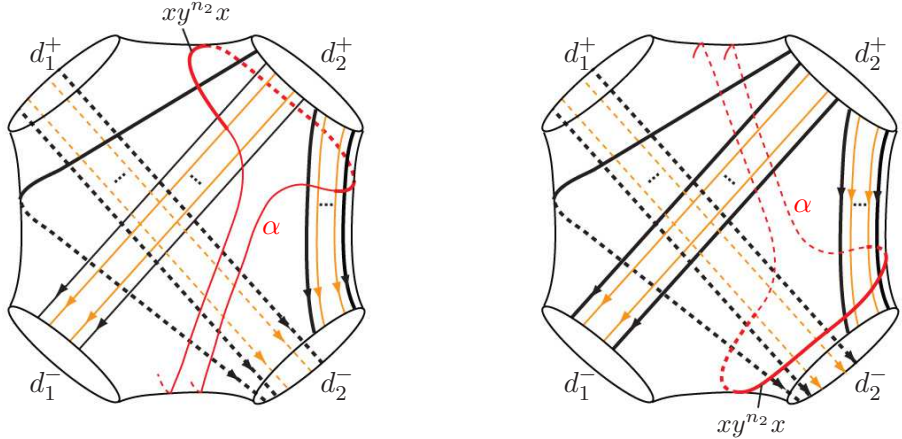}
\linethickness{3pt}
\put(10, 140){$d^+_1$}
\put(10, 15){$d^-_1$}
\put(120, 140){$d^+_2$}
\put(120, 15){$d^-_2$}
\put(98, 75){\color{red} $\alpha$}
\put(60, 156){\small $x y^{n_2}x$}
\put(212, 140){$d^+_1$}
\put(212, 15){$d^-_1$}
\put(322, 140){$d^+_2$}
\put(322, 15){$d^-_2$}
\put(299, 83){\color{red} $\alpha$}
\put(270, 0){\small $x y^{n_2}x$}
\end{overpic}
\captionof{figure}{The bold part of $\alpha$ determines the word $x y^{n_2} x$.}
\label{fig:subarcs_1}
\end{center}
By Lemma \ref{lem:key3} this implies that $F$ is not a primitive disk, whence a contradiction.

\medskip

\noindent \textit{Case }2. $r = s$.

In this case the disk from surgery on $D_1 \cup D_2$ along $F$ is $D_0$, which is impossible by the assumption of the induction.

\medskip

\noindent \textit{Case }3. $s > r > 0$.
Suppose first that $1/2 > r/s$.
If $r$ is odd, with a suitable choice of an orientation of $\alpha$ the word corresponding to the arc $\alpha$ contains the term $x y^{-k_3 + k_6} x^{-1}$
after canceling pairs of $y$ and $y^{-1}$ if necessary.
See the left-hand side in Figure \ref{fig:subarcs_2}.
\begin{center}
\begin{overpic}[width=12cm,clip]{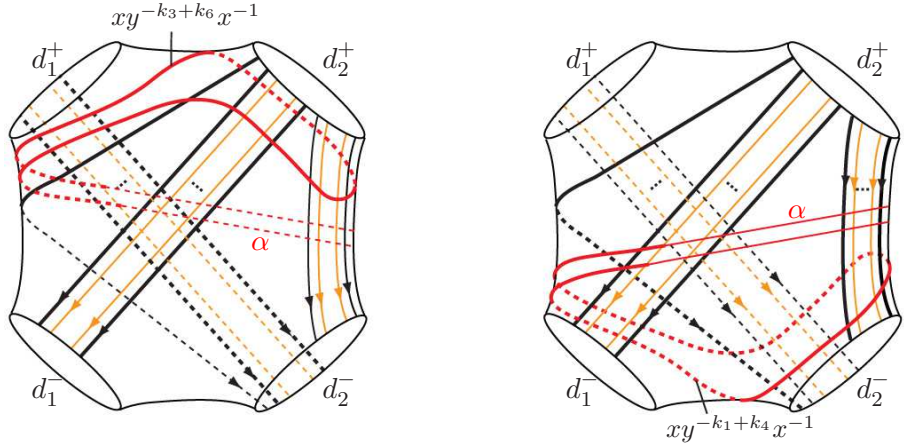}
\linethickness{3pt}
\put(10, 140){$d^+_1$}
\put(10, 15){$d^-_1$}
\put(120, 140){$d^+_2$}
\put(120, 15){$d^-_2$}
\put(93, 71){\color{red} $\alpha$}
\put(40, 156){\small $x y^{-k_3 + k_6} x^{-1}$}
\put(212, 140){$d^+_1$}
\put(212, 15){$d^-_1$}
\put(322, 140){$d^+_2$}
\put(322, 15){$d^-_2$}
\put(296, 83){\color{red} $\alpha$}
\put(250, 0){\small $x y^{-k_1 + k_4} x^{-1}$}
\end{overpic}
\captionof{figure}{Left: the bold part of $\alpha$ determines the word $x y^{-k_3 + k_6} x^{-1}$.
Right: the bold part of $\alpha$ determines the word $x y^{-k_1 + k_4} x^{-1}$.}
\label{fig:subarcs_2}
\end{center}
By the solution of the equation $(\ref{eq:solution of the linear equation})$,
we have $- k_3 + k_6 = - \bar{q} + n_1 \geq 1$.
By Lemma \ref{lem:key2} this implies that $F$ is not a primitive disk, whence a contradiction.
If $r$ is even, with a suitable choice of an orientation of $\alpha$ the word corresponding to the arc $\alpha$ contains the term $x y^{-k_1 + k_4} x^{-1}$
after canceling pairs of $y$ and $y^{-1}$ if necessary.
See the right-hand side in Figure \ref{fig:subarcs_2}.
By the solution of the equation $(\ref{eq:solution of the linear equation})$,
we have $- k_1 + k_4 = \bar{q} - n_1 \leq -1$.
Thus, by Lemma \ref{lem:key2}, $F$ is not a primitive disk, a contradiction.

Next suppose that $r/s > 1/2$.
If $r$ is odd, with a suitable choice of an orientation of $\alpha$ the word corresponding to the arc $\alpha$ contains the term $x y^{k_1 -k_4} x^{-1}$ after canceling pairs of $y$ and $y^{-1}$ if necessary.
See the left-hand side in Figure \ref{fig:subarcs_3}.
\begin{center}
\begin{overpic}[width=12cm,clip]{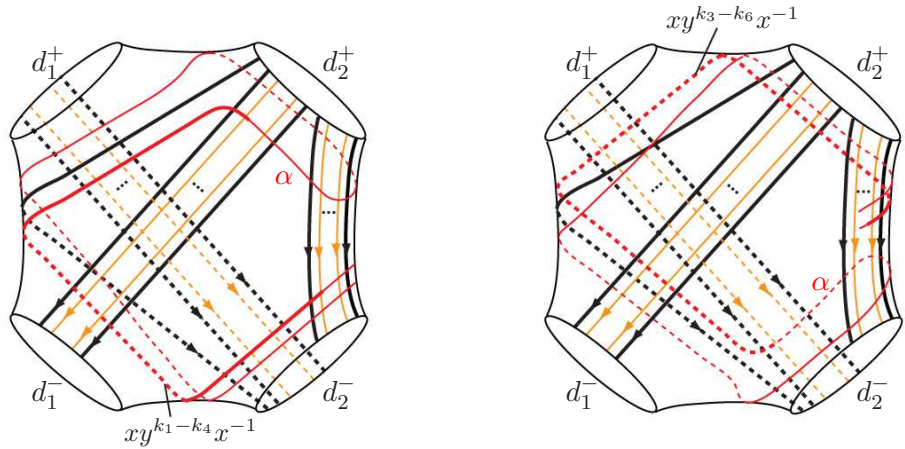}
\linethickness{3pt}
\put(10, 140){$d^+_1$}
\put(10, 15){$d^-_1$}
\put(120, 140){$d^+_2$}
\put(120, 15){$d^-_2$}
\put(102, 97){\color{red} $\alpha$}
\put(45, 0){\small $x y^{k_1 - k_4} x^{-1}$}
\put(212, 140){$d^+_1$}
\put(212, 15){$d^-_1$}
\put(322, 140){$d^+_2$}
\put(322, 15){$d^-_2$}
\put(305, 57){\color{red} $\alpha$}
\put(250, 156){\small $x y^{k_3 - k_6} x^{-1}$}
\end{overpic}
\captionof{figure}{Left: the bold part of $\alpha$ determines the word $x y^{k_1 - k_4} x^{-1}$.
Right: the bold part of $\alpha$ determines the word $x y^{k_3 - k_6} x^{-1}$.}
\label{fig:subarcs_3}
\end{center}
As above, we have $k_1 - k_4 = - \bar{q} + n_1 \geq 1$, whence a contradiction by Lemma \ref{lem:key2}.
If $r$ is even, with a suitable choice of an orientation of $\alpha$ the word corresponding to the arc $\alpha$ contains the term $x y^{k_3 - k_6} x^{-1}$
after canceling pairs of $y$ and $y^{-1}$ if necessary.
See the right-hand side in Figure \ref{fig:subarcs_3}.
As above we have $k_3 - k_6 = \bar{q} - n_1 \leq 1$, whence a contradiction by Lemma \ref{lem:key2}.

\medskip

\noindent \textit{Case }4. $0 > r$.
In this case with a suitable choice of an orientation of $\alpha$ the word corresponding to the arc $\alpha$ contains the term $x y^{k_1 -k_5 + k_7} x^{-1}$
after canceling pairs of $y$ and $y^{-1}$ if necessary.
See Figure \ref{fig:subarcs_4}.

\medskip

\begin{center}
\begin{overpic}[width=12cm,clip]{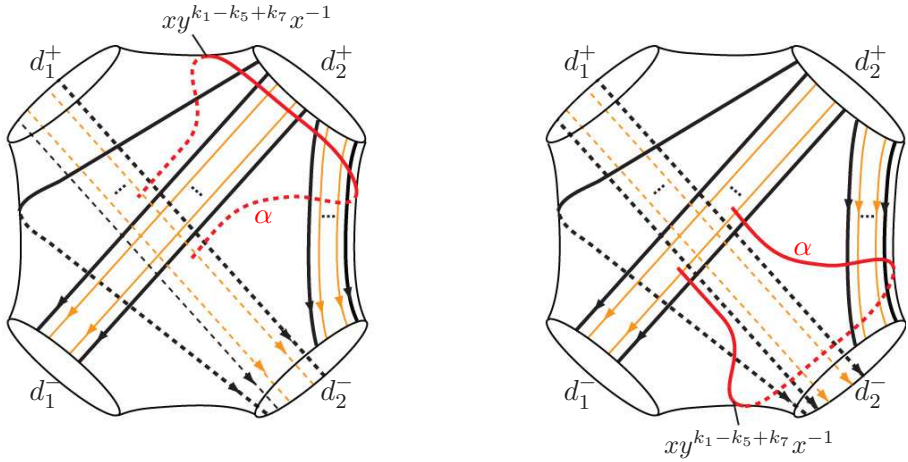}
\linethickness{3pt}
\put(10, 140){$d^+_1$}
\put(10, 15){$d^-_1$}
\put(120, 140){$d^+_2$}
\put(120, 15){$d^-_2$}
\put(95, 82){\color{red} $\alpha$}
\put(60, 156){\small $x y^{k_1 -k_5 + k_7} x^{-1}$}
\put(212, 140){$d^+_1$}
\put(212, 15){$d^-_1$}
\put(322, 140){$d^+_2$}
\put(322, 15){$d^-_2$}
\put(299, 70){\color{red} $\alpha$}
\put(250, -5){\small $x y^{k_1 -k_5 + k_7} x^{-1}$}
\end{overpic}
\captionof{figure}{The bold part of $\alpha$ determines the word $x y^{k_1 -k_5 + k_7} x^{-1}$.}
\label{fig:subarcs_4}
\end{center}
As above, we have $k_1 -k_5 + k_7 = - \bar{q} + n_2 \leq -1$, whence a contradiction by Lemma \ref{lem:key2}.
\end{proof}

\begin{lemma}
\label{lem:primitive disks in a bridge}
Let $\mathcal{B}_{ \{ D , E \} }$ be a bridge connecting primitive disks $D$ and $E$.
Then $V_D = \cl(V- \Nbd(D))$ is not isotopic to $V_E = \cl(V- \Nbd(E))$ in $L(p,q)$.
In particular,
$D$ is of $(p,q)$-type if and only if $E$ is of $(p,q')$-type.
\end{lemma}
\begin{proof}
Choose $E'$, $E_0$ and $E_0'$ as in the proof of  Lemma \ref{lem:principal pair of a bridge}.
Assigning symbols $x$ and $y$ to oriented $\partial E'$ and $\partial E'_0$ respectively,
any oriented simple closed curve on $\partial W$ represents an element of the free group $\pi_1 (W) = \langle x, y \rangle$.
Note that the natural projection
\[ \varphi : H_1(W; \Integer) = \Integer [x] \oplus \Integer [y] \to H_1(L(p,q); \Integer ) = \Integer / p \Integer \]
induced by the inclusion $W \hookrightarrow L(p,q)$ satisfies $\varphi ([x]) = 0$ and $\varphi ([y]) = \pm 1$.

By Lemmas \ref{lem:principal pair of a bridge} and \ref{lem:words for disks in a bridge},
the bridge $\mathcal{B}_{ \{ D , E \} }$ is obtained as in Lemma \ref{lem:construction of a bridge}.
In particular,
by the argument of Lemma \ref{lem:construction of a bridge},
the circle $\partial D$ determines the element of the form $(x y^{\bar{q}})^k xy^{ \bar{q} \pm 1}$ for some $k \in \Natural$
while $\partial E$ determines $x$ in $\pi_1 (W) = \langle x, y \rangle$.

Consider the exteriors $W_D := \cl( L(p,q ) - V_D)$ and $W_E := \cl( L(p,q ) - V_E)$
of $V_D$ and $V_E$ in $L(p,q)$, and
let $l_D$ and $l_E$ be the cores of those solid tori $W_D$ and $W_E$ respectively.
It suffices to show that their homology classes $[l_D]$ and $[l_E]$ differs in $H_1 (L(p,q) , \Integer)$.
We regard $H_1(W_D; \Integer)$ and $H_1(W_E ; \Integer)$ as subgroups of $H_1(W ; \Integer)$ in a natural way.
It is then easy to see from the construction that $H_1 (W_E ; \Integer) = (\Integer[x] \oplus \Integer [y] )/ \Integer [x] = \Integer [y]$, which implies that $\pm [l_E] = \pm \varphi ([y]) = \pm 1 \in \Integer / p \Integer$.
On the other hand, we have  $H_1 (W_D ; \Integer) = (\Integer [x] \oplus \Integer [y] )/  \Integer ( (k+1) ([x] + \bar{q} [y]) \pm [y] )
= \Integer ([x] + \bar{q} [y] )$, which implies that $\pm [l_D] = \pm \varphi ( [x] + \bar{q}  [y]) = \pm \bar{q} \in \Integer / p \Integer$.
Since $\bar {q} \neq \pm 1 \pmod p$, we have $[l_D] \neq [l_E]$ in $H_1(L(p , q) ; \Integer)$.
This completes the proof.
\end{proof}

We recall that the primitive disk complex $\mathcal P(V)$ consists of infinitely many tree components, and given any vertex $E$ of $\mathcal P(V)$, there are infinitely many shells $\mathcal S_E = \{E_0, E_1, \ldots, E_p\}$ centered at $E$ by the choice of a semiprimitive disk $E_0$ and the choice of a primitive disk $E_1$.
Thus, by Lemma \ref{lem:construction of a bridge}, for each vertex $E$ of $\mathcal P(V)$ there are infinitely many bridges having $E$ as its end vertex.
Further, we have the following description of the bridges.

\begin{lemma}
\label{lem:uniqueness of a bridge}
Any two bridges are isomorphic to each other. Any two bridges are either disjoint from each other or intersect only in an end vertex.
\end{lemma}

\begin{proof}
By Lemma \ref{lem:construction of a bridge}, given any $(p, q)$-shell $\mathcal S_E = \{E_0, E_1, \ldots, E_p\}$, there exists a bridge $\mathcal B_{\{D, E\}} = \{\Delta_1, \Delta_2, \ldots, \Delta_n\}$ such that $\Delta_1 = \{E, E_m, E_{m+1}\}$.
By Lemma \ref{lem:words for disks in a bridge}, such a bridge is unique.
That is, if any bridge has $\Delta_1$ as its first $2$-simplex, then it is exactly $\mathcal B_{\{D, E\}}$.
Given any other bridge $\mathcal B_{\{ \bar{D}, \bar{E}\}}$, by Lemma \ref{lem:primitive disks in a bridge}, we may assume that one of $\bar{D}$ and $\bar{E}$, say $\bar{E}$, is of $(p, q)$-type.
By Lemma \ref{lem:principal pair of a bridge}, there exists a $(p, q)$-shell $\mathcal S_{\bar{E}}$ containing the first $2$-simplex of $\mathcal B_{\{ \bar{D}, \bar{E}\}}$.
Thus $\mathcal B_{\{ \bar{D}, \bar{E}\}}$ is isomorphic to $\mathcal B_{\{D, E\}}$.
The second statement is also a direct consequence of Lemmas \ref{lem:principal pair of a bridge} and \ref{lem:words for disks in a bridge}.
\end{proof}

\subsection{The tree of trees}
We again assume that $(V, W; \Sigma)$  is the genus-$2$ Heegaard splitting of a lens space $L = L(p, q)$ with $1 \leq q \leq p/2$ and $p \not\equiv \pm 1 \pmod q$.
So far, we have seen that for any vertex $E$ of $\mathcal P(V)$, there are infinitely many bridges of which $E$ is an end vertex, and further:
\begin{itemize}
\item any two bridges are isomorphic to each other,
\item any two bridges are disjoint from each other or intersect only in an end vertex,
\item any bridge connects exactly two tree components of $\mathcal P(V)$, and
\item any two tree components of $\mathcal P(V)$ is connected by at most a single bridge.
\end{itemize}

\bigskip

\begin{center}
\begin{overpic}[width=12cm,clip]{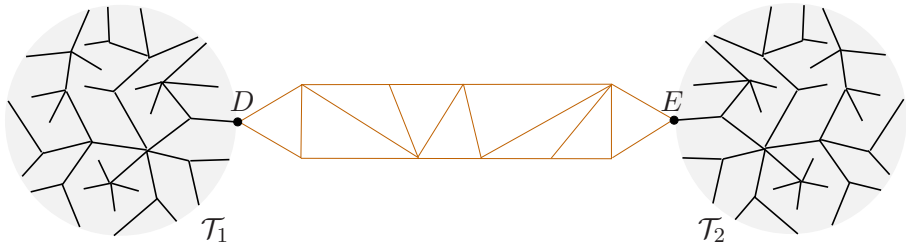}
\linethickness{3pt}
\put(85, 48){$D$}
\put(248, 48){$E$}
\put(74, 0){$\mathcal T_1$}
\put(262, 0){$\mathcal T_2$}
\end{overpic}
\captionof{figure}{The unique bridge connecting the tree components $\mathcal{T}_1$ and $\mathcal{T}_2$ of $\mathcal P(V)$.}
\label{fig:bridge}
\end{center}

\bigskip

Thus, by shrinking each of the tree components of $\mathcal P(V)$ to a vertex, and each of the bridges to an edge connecting the two end vertices, we have a tree $\mathcal T^{\mathcal T} (V)$, which we call the {\it ``tree of trees''} for the splitting $(V, W; \Sigma)$.
We note that each vertex of $\TT (V)$ has infinite valency.
Since the action of the Goeritz group $\mathcal{G}$ of the splitting $(V, W; \Sigma)$
preserves the set of vertices of $\mathcal P(V)$ and the set of bridges,
the action of $\mathcal{G}$ on $\mathcal{D}(V)$ naturally induces
a simplicial action of $\mathcal{G}$ on $\TT(V)$.

\section{Example: the lens space $L(12, 5)$}
\label{sec:first_example}

Let $L(p, q)$ be a lens space with $1 \leq q \leq p/2$.
Let $(V, W; \Sigma)$ be the genus-$2$ Heegaard splitting of $L(p, q)$.
As we have seen in Theorem \ref{thm:contractibility},
the primitive disk complex $\mathcal P(V)$
is contractible if and only if $p \equiv \pm 1 \pmod{q}$.
This implies that the primitive disk complex $\mathcal P(V)$
is contractible for every lens space $L(p,q)$ with $p < 12$.
In this section, we focus on the lens space
$L(12,5)$: the ``smallest" lens space with disconnected primitive disk complex.
Recall that in this case, the primitive disk complex $\mathcal{P}(V)$ consists
of infinitely many tree components.
We describe the combinatorial structure of the primitive disk complex and the briges,
and provide an idea to obtain a presentation of the Goeritz group.
The argument in this section will be soon generalized for
every $L(p,q)$, $1 \leq q \leq p/2$, with $p \not\equiv \pm 1 \pmod{q}$ in
the next section.

Let $(V, W; \Sigma)$ be the genus-$2$ Heegaard splitting of $L(12, 5)$.
Let $E$ be a primitive disk in $V$.
Since $5^2 \equiv 1 \pmod {12}$, $q = q'$ in this case
(recall that $q'$ is the unique integer satisfying
$1 \leq q' \leq p/2$ and $qq' \equiv \pm 1 \pmod p$), $E$ is always of $(12,5)$-type.
Let $\mathcal S_E = \{E_0, E_1, \ldots, E_{12}\}$ be a shell centered at $E$.
Let $E'$ be a unique dual disk of $E$ disjoint from $E_0$, and
let $E_0'$ be a unique semiprimitive disk in $W$ disjoint from $E$.
Lemma \ref{lem:construction of a bridge} says that
there exists a bridge
$\mathcal{B}_{\{ D , E \}} = \{ \Delta_1 , \Delta_2 , \ldots , \Delta_n \}$
connecting $E$ and a certain primitive disk $D$ in $V$ with $\Delta_1 = \{ E, E_2 , E_3 \}$.
We can easily construct that bridge as follows.
The left-hand side in Figure \ref{fig:construction_of_a_bridge}
depicts the 4-holed sphere $\partial V$ cut off by $E \cup E_2$.
We denote by $e^\pm$ and $e_2^\pm$ its boundary
circles coming from $\partial E$ and $\partial E_2$ respectively.

\begin{center}
\begin{overpic}[width=14cm,clip]{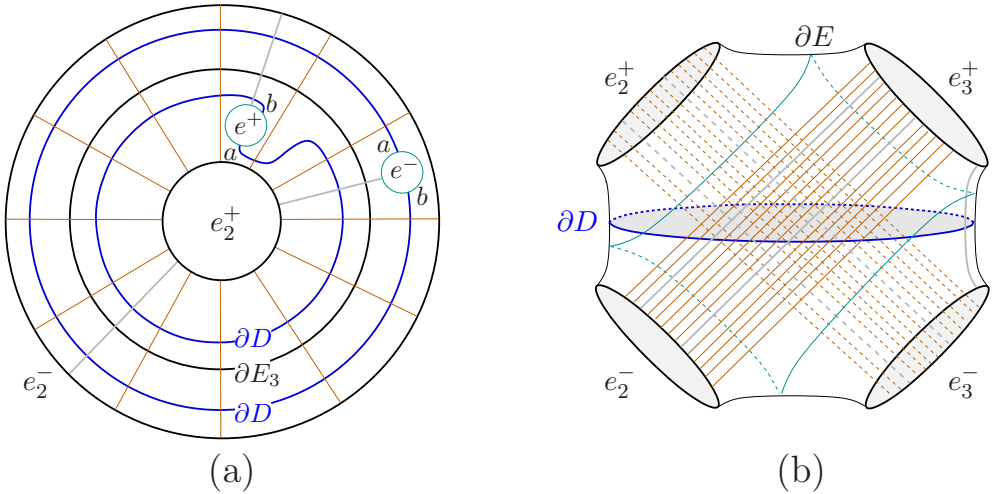}
 \linethickness{3pt}
  \put(100, 127){\small $e^+$}
  \put(159, 110){\small $e^-$}
  \put(91, 90){$e_2^+$}
  \put(20, 30){$e_2^-$}
  \put(96, 117){\small $a$}
  \put(112, 135){\small $b$}
  \put(154, 120){\small $a$}
  \put(169, 100){\small $b$}
  \put(100, 18){\small {\color{blue}$\partial D$}}
  \put(100, 33){\small $\partial E_3$}
  \put(100, 46){\small {\color{blue}$\partial D$}}
  \put(240, 145){$e_2^+$}
  \put(240, 30){$e_2^-$}
  \put(370, 145){$e_3^+$}
  \put(370, 30){$e_3^-$}
  \put(312, 160){$\partial E$}
  \put(222, 90){{\color{blue}$\partial D$}}
  \put(90, -5){\Large (a)}
  \put(305, -5){\Large (b)}
\end{overpic}
\captionof{figure}{(a) The $4$-holed sphere $\partial V$ cut off by $E_2 \cup E$. (b) The $4$-holed sphere $\partial V$ cut off by $E_2 \cup E_3$.}
\label{fig:construction_of_a_bridge}
\end{center}

\bigskip

In Figure \ref{fig:construction_of_a_bridge} (a), $\partial E_0'$ separates the 4-holed sphere
into 12 rectangles, and $\partial E_2$ appears as 3 segments.
Assigning symbols $x$ and $y$ to $\partial E'$ and $\partial E'_0$ with
appropriate orientations respectively,
the simple closed curves $\partial E_2$ and $\partial E_3$
(with appropriate orientations) represents
the elements $xy^5 xy^{12}$ and $(xy^5)^2 xy^2$ respectively,
of the free group $\pi_1 (W) = \langle x, y \rangle$.
Let $D$ be a disk whose boundary circle is described in the figure.
The disk $D$ intersects $E$ transversely in an arc, and
the simple closed curve $\partial D$ represents the element
$(xy^5)^4 xy^4$ in $\pi_1 (W)$.
This is a primitive element in $\pi_1 (W)$, see e.g. Osborne-Zieschang \cite{OZ}.
Hence, $D$ is a primitive disk in $V$ by Lemma \ref{lem:primitive_element}.
We can see the disk $D$ using Figure \ref{fig:construction_of_a_bridge} (b) as well.
The figure illustrates the 4-holes sphere
$\partial V$ cut off by $E_2 \cup E_3$ instead of $E \cup E_2$.
As a consequence, setting $\Delta_1 = \{ E, E_2, E_3 \}$, $\Delta_2 = \{ E_2, E_3, D \}$,
$\{ \Delta_1, \Delta_2 \}$ forms the bridge $\mathcal{B}_{ \{ D, E \} }$
connecting $D$ and $E$, see Figure \ref{fig:bridge_L12_5}.

\bigskip

\begin{center}
\begin{overpic}[width=3.5 cm,clip]{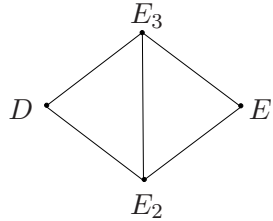}
 \linethickness{3pt}
  \put(-1, 36){$D$}
  \put(90, 36){$E$}
  \put(45, 72){$E_3$}
  \put(45, 0){$E_2$}
\end{overpic}
\captionof{figure}{The bridge $\mathcal{B}_{ \{ D, E \} }$.}
\label{fig:bridge_L12_5}
\end{center}

\bigskip

We note that by Lemmas \ref{lem:construction of a bridge},
\ref{lem:principal pair of a bridge} and \ref{lem:uniqueness of a bridge},
every bridge in $\mathcal{D}(V)$ for $L(12,5)$
is constructed in this way.
In particular, every bridge is isomorphic to
$\mathcal{B}_{ \{ D, E \} }$ above.
The subcomplex
$\mathcal{P}(V) \cup (\bigcup_{\mathcal{B}} \mathcal{B})$ of $\mathcal D(V)$, where
$\mathcal{B}$ runs over all bridges in $\mathcal{D}(V)$,
is connected.
Moreover, since the dual of the disk complex $\mathcal{D}(V)$ is a tree,
$\mathcal{P}(V) \cup (\bigcup_{\mathcal{B}} \mathcal{B})$
is contractible.
This allows us to define a tree $\TT(V)$ as follows.
The vertices of $\TT(V)$ are the components of $\mathcal{P} (V)$,
and two vertices $\mathcal{T}_1$ and $\mathcal{T}_2$ of
$\TT(V)$ span an edge if and only
if there exists a bridge $\mathcal{B}_{\{ D, E \}}$ with
$E \in \mathcal{T}_1$ and $D \in \mathcal{T}_2$.
See Figure \ref{fig:TT_L12_5}.
Note that $\TT(V)$ is not locally finite.

The action of the Goeritz group $\mathcal{G}$ on the disk complex
$\mathcal{D}(V)$ induces, in a natural way, a simplical action
of $\mathcal{G}$ on $\TT(V)$ as well.
Let $\mathcal{B}_{\{D, E\}}$ be a bridge (so an edge of $\TT(V)$) connecting two tree components
$\mathcal{T}_1$ and $\mathcal{T}_2$ with $E \in \mathcal{T}_1$ and $D \in \mathcal{T}_2$.
In the next section, we will show the following in a general setting:
\begin{enumerate}
\item
The action of $\mathcal{G}$
on the set of vertices $($edges, respectively$)$ of $\TT(V)$ is transitive
(cf. Lemma \ref{lem:transitivity on the vertices of TT}).
\item
There exists an element $\tau$ of $\mathcal{G}$ that
preserves the bridge $\mathcal{B}_{\{ D , E \} }$ but
exchanges $D$ and $E$
(cf. Lemma \ref{lem:stabilizers of edges of TT}
(\ref{item:stabilizers of edges of TT (1)})).
\end{enumerate}

Let $(\TT)'$ be the first barycentric subdivision
of $\TT$.
By Lemmas \ref{lem:transitivity on the vertices of TT} and \ref{lem:stabilizers of edges of TT},
the Goeritz group $\mathcal{G}$ acts on the set of edges of $(\TT)'$
transitively without inverting edges, and the two endpoints of each edge belong to different
orbits of vertices under the action of $\mathcal{G}$.
Hence, the quotient of $(\TT)'$ by the action of $\mathcal{G}$
is a single edge with two vertices.
Now we can use the Bass-Serre theory (cf. Theorem \ref{thm:theorem by Brown}) for this group action,
and hence, we can express $\mathcal{G}$ as the following amalgamated free product:
\[ \mathcal{G}_{\{ \mathcal{T}_1 \}} *_{\mathcal{G}_{ \{ \mathcal{T}_1 , \mathcal{T}_2 \} }}
\mathcal{G}_{\{ \mathcal{T}_1 \cup \mathcal{T}_2 \}} . \]
Here by $\mathcal{G}_{\{ \mathcal{T}_1 \}}$, $\mathcal{G}_{\{ \mathcal{T}_1 \cup \mathcal{T}_2 \}}$ and
$\mathcal{G}_{\{ \mathcal{T}_1 , \mathcal{T}_2 \}} $ we mean
the isotropy subgroups of $\mathcal{G}$ with respect to
the vertex $\mathcal{T}_1$, the unordered pair of $\mathcal{T}_1$ and $\mathcal{T}_2 $, and
the ordered pair of $\mathcal{T}_1$ and $\mathcal{T}_2 $ respectively.
In the next section, we give finite presentations of these 3 subgroups
(cf. Lemmas
\ref{lem:stabilizers of vertices of TT} and \ref{lem:stabilizers of edges of TT}).
In this way, obtain a presentation of the Goeritz group $\mathcal{G}$.

\begin{center}
\begin{overpic}[width=14cm,clip]{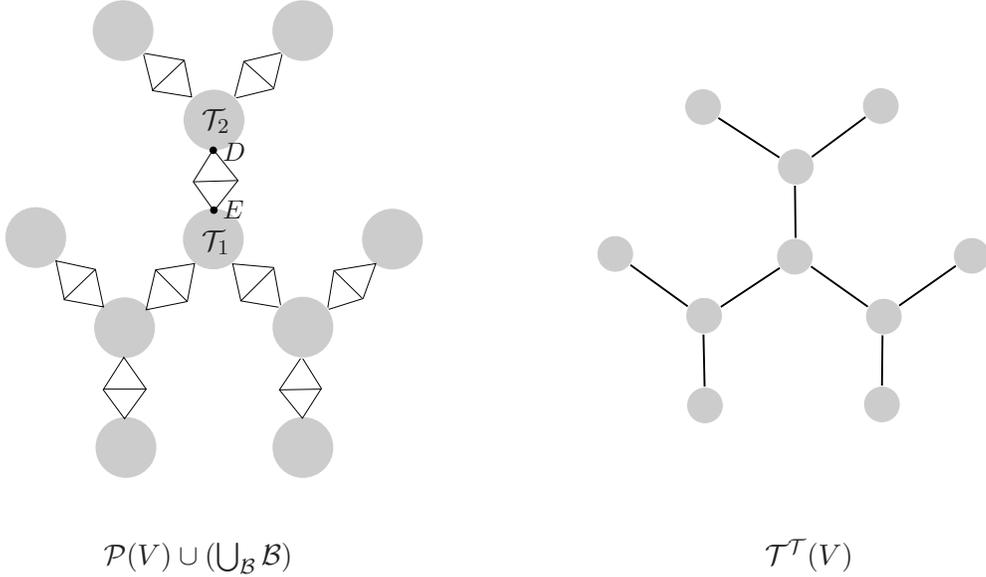}
 \linethickness{3pt}
  \put(87, 185){$\mathcal T_2$}
  \put(87, 139){$\mathcal T_1$}
  \put(95, 173){\small $D$}
  \put(95, 151){\small $E$}
  \put(50, 20){$\mathcal{P}(V) \cup (\bigcup_{\mathcal{B}} \mathcal{B})$}
  \put(300, 20){$\TT(V)$}
\end{overpic}
\captionof{figure}{The subcomplex $\mathcal{P}(V) \cup (\bigcup_{\mathcal{B}} \mathcal{B})$ of $\mathcal{D}(V)$, and the tree $\TT(V)$.}
\label{fig:TT_L12_5}
\end{center}

\section{The mapping class groups of the genus-$2$ Heegaard splittings for lens spaces}
\label{sec:main_section}

Let $\mathcal G$ be the genus-$2$ Goeritz group of a lens space $L(p,q)$ with $1 \leq q \leq p/2$, and let $(V, W; \Sigma)$ be a genus-$2$ Heegaard splitting of $L(p,q)$.
Throughout the section, we will assume that $p \not\equiv \pm 1 \pmod q$, and we will fix the followings:
\begin{itemize}
\item
A primitive disk $E$ of $(p,q)$-type in $V$;
\item
A $(p, q)$-shell $\mathcal{S}_E = \{E_0, E_1, \ldots, E_p \}$ centered at $E$;
\item
The unique $(p, q')$-shell $\mathcal{S}_C = \{ C_0, C_1, \ldots, C_p \}$
centered at $C = E_{q'}$
such that $E=C_q$ (cf. Lemma \ref{lem:sequence} (3));
\item
The component $\mathcal{T}_1$ of $\mathcal{P}(V)$, which is a tree, that contains
$E$ (and so $C$);
\item
The unique bridge $\mathcal{B}_{\{ D, E \}} = \{ \Delta_1 , \Delta_2, \ldots, \Delta_n \}$,
with $\Delta_1 = \{E, E_m, E_{m+1}\}$ where $m$ is the integer satisfying $p = qm + r$ for $2 \leq r \leq q-2$. Note that $D$ is of $(p,q')$-type;
\item
The component $\mathcal{T}_2$ of $\mathcal{P}(V)$ that contains
$D$.
\end{itemize}
See Figure \ref{fig:setting}.

\begin{center}
\begin{overpic}[width=13cm,clip]{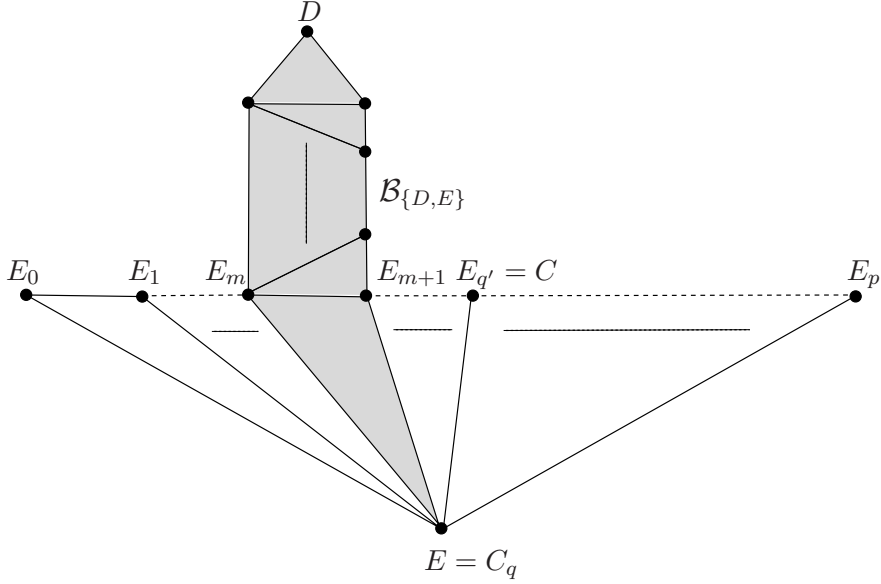}
 \linethickness{3pt}
  \put(132, 218){$D$}
  \put(163, 150){\large $\mathcal B_{\{D, E\}}$}
  \put(22, 120){$E_0$}
  \put(68, 120){$E_1$}
  \put(97, 120){$E_m$}
  \put(162, 120){$E_{m+1}$}
  \put(192, 120){$E_{q'} = C$}
  \put(340, 120){$E_p$}
  \put(180, 10){$E = C_{q}$}
\end{overpic}
\captionof{figure}{The primitive disks $E$, $C$, and $D$.}
\label{fig:setting}
\end{center}

We use the above four primitive disks $E$, $C$, $E_1$, $C_1$ to describe the orbits of the action of
the Goeritz group $\mathcal G$ to the set of primitive pairs.

\begin{lemma}[Lemmas 5.2 and 5.3 \cite{CK15b}]
\label{lem:number of orbits of primitive disks and pairs}

\begin{enumerate}
\item
If $q^2 \equiv 1 \pmod p$,
the action of the Goeritz group $\mathcal{G}$ on
the set of vertices of the primitive disk complex $\mathcal P(V)$ is transitive.
Further, the action of $\mathcal{G}$
on the set of edges of $\mathcal P(V)$ has
exactly $2$ orbits $\mathcal{G} \cdot \{ E, C \}$ and $\mathcal{G} \cdot \{ E, E_1 \}$.
The two end points of each of the edges $\{ E, C \}$ and $\{ E, E_1 \}$ can be exchanged
by the action of $\mathcal{G}$.
\item
If $q^2 \not\equiv 1 \pmod p$, the action of $\mathcal{G}$ on
the set of vertices of $\mathcal P(V)$ has exactly two orbits
$\mathcal{G} \cdot \{E\}$ and $\mathcal{G} \cdot \{C\}$.
Further, the action of $\mathcal{G}$ on the set of edges of $\mathcal P(V)$ has
exactly $3$ orbits $\mathcal{G} \cdot \{ E, C \}$,
$\mathcal{G} \cdot \{ E, E_1 \}$ and $\mathcal{G} \cdot \{ C, C_1 \}$.
The two end points of each of the edges $\{ E, E_1 \}$ and $\{ C, C_1 \}$
can be exchanged by the action of $\mathcal{G}$ whereas those of $\{ E, C \}$ cannot.
\end{enumerate}
\end{lemma}

\begin{lemma}
\label{lem:transitivity on shells}
Let $A$, $B$ be primitive disk in $V$.
Let $\mathcal{S}_A = \{ A_0, A_1, \ldots , A_p \}$ and
$\mathcal{S}_B = \{ B_0, B_1, \ldots , B_p \}$ be
shells centered at $A$ and $B$ respectively.
Suppose that there exists an element of $\mathcal{G}$ that
maps $B$ to $A$.
Then there exists an element $\varphi$ of $\mathcal{G}$
satisfying $\varphi (B) = A$ and $\varphi (B_i) = A_i$ $(i \in \{ 0, 1, \ldots , p \})$.
\end{lemma}

\begin{lemma}
\label{lem:transitivity on the vertices of TT}
The action of $\mathcal{G}$
on the set of vertices $($edges, respectively$)$ of $\TT(V)$ is transitive.
\end{lemma}
\begin{proof}
If $q^2 \equiv 1 \pmod p$, the action of $\mathcal{G}$ is transitive
on the set of primitive disks by Lemma \ref{lem:number of orbits of primitive disks and pairs} (1),
which implies that $\mathcal{G}$ acts transitively on the set of vertices of $\TT(V)$.
If $q^2 \not\equiv 1 \pmod p$, each connected component of $\mathcal{P}(V)$ contains
vertices of both $(p,q)$ and $(p,q')$-types.
Thus, it follows from Lemma \ref{lem:number of orbits of primitive disks and pairs} (2)
that $\mathcal{G}$ acts transitively on the set of vertices of $\TT(V)$.

Let $\mathcal{B}_{\{ \bar{D} , \bar{E} \}} =
\{ \bar{\Delta}_1 , \bar{\Delta}_2 , \ldots , \bar{\Delta}_n \}$ be an arbitrary bridge.
By Lemma \ref{lem:primitive disks in a bridge}, we can assume without
loss of generality that
$\bar{D} \in \mathcal{G} \cdot D$,
$\bar{E} \in \mathcal{G} \cdot E$.
Let $\{ E_* , E_{**} \}$ ($\{ \bar{E}_* , \bar{E}_{**} \}$, respectively)
be the principal pair of $E$ ($\bar{E}$, respectively) with respect to
$D$ ($\bar{D}$, respectively).
Then by Lemma \ref{lem:principal pair of a bridge},
we have $\{ E_* , E_{**} \} = \{ E_m , E_{m+1} \}$, and also
there exists a $(p, q)$-shell
$\mathcal S_{\bar{E}} = \{\bar{E}_0, \bar{E}_1, \ldots, \bar{E}_p\}$ centered at
$\bar{E}$ such that $\{ \bar{E}_* , \bar{E}_{**} \} = \{ \bar{E}_m , \bar{E}_{m+1} \}$.
By Lemma \ref{lem:transitivity on shells}, there exists an element $\varphi$
of the Goeritz group $\mathcal{G}$ satisfying
$g(\bar{E}) = E$ and $g(\bar{E}_i) = E_i$ for $i \in \{ 0 , 1 , \ldots , p \}$.
Then $\varphi$ maps the bridge $\mathcal{B}_{\{\bar{D}, \bar{E}\}}$ to
another bridge
$\varphi (\mathcal{B}_{\{\bar{D}, \bar{E}\}}) = \mathcal{B}_{\{\varphi (\bar{D}), E \}}
= \{ \varphi (\bar{\Delta}_1) , \varphi (\bar{\Delta}_2) , \ldots , \varphi (\bar{\Delta}_n) \}$.
Since $\Delta_1 = \{ E, E_m , E_{m+1} \} = \{ \varphi (\bar{E}) , \varphi (\bar{E}_m), \varphi (\bar{E}_{m+1}) \}
= \varphi (\bar{\Delta}_1)$ we have
$\mathcal{B}_{\{ D , E \}} =  \varphi (\mathcal{B}_{\{ \bar{D} , \bar{E} \}} )$ by
Lemma \ref{lem:uniqueness of a bridge}.
This completes the proof.
\end{proof}

To obtain a finite presentation of the Goeritz group $\mathcal{G}$, we use the
following well-known theorem.
\begin{theorem}[Serre \cite{S}]
\label{thm:theorem by Brown}
Suppose that a group $G$ acts on a tree $\mathcal{T}$
without inversion on the edges.
If there exists a subtree $\mathcal{L}$ of $\mathcal{T}$
such that every vertex $($every edge, respectively$)$ of $\mathcal{T}$ is equivalent modulo $G$
to a unique vertex $($a unique edge, respectively$)$ of $\mathcal{L}$.
Then $G$ is the free product of the isotropy groups $G_v$ of the
vertices $v$ of $\mathcal{L}$, amalgamated along the isotropy groups $G_e$
of the edges $e$ of $\mathcal{L}$.
\end{theorem}

In the following we will denote by $\mathcal{G}_{\{X_1 , X_2 , \ldots , X_k \}}$ the subgroup of
the genus-2 Goeritz group $\mathcal{G}$ consisting of elements that preserve each of
$X_1$, $X_2 , \ldots , X_k$ setwise, where each $X_i$ will be a subcomplex of $\mathcal{D}(V)$.

\begin{lemma}
\label{lem:stabilizer of a primitive disks and pairs}
\begin{enumerate}
\item
Let $A$ be an arbitrary primitive disk in $V$.
Then we have
$\mathcal{G}_{\{ A \}} =
\langle \alpha \mid \alpha^2 \rangle
\oplus \langle \beta, \gamma \mid
{\gamma}^2 \rangle$, where
$\alpha$ is the hyperelliptic involution of both
$V$ and $W$, $\beta$ is the half-twist along a reducing sphere, and
$\gamma$ exchanges two disjoint dual disks of $A$ as described in Figure ${\rm \ref{fig:isotropy_of_a_primitive_disk}}$.
\item
Let $\{A, B \}$ be an edge of the primitive disk complex $\mathcal P(V)$.
Then we have $\mathcal{G}_{ \{ A, B \} } = \langle \alpha \mid \alpha^2 \rangle$.
If the two end points of the edge $\{ A, B \}$
can be exchanged by the action of $\mathcal{G}$,
then we have
$\mathcal{G}_{ \{ A \cup B \} } =
\langle \alpha \mid \alpha^2 \rangle \oplus
\langle \sigma \mid \sigma^2 \rangle$, where $\sigma$ is an element of $\mathcal G$ exchanging $A$ and $B$.
Otherwise, we have
$\mathcal{G}_{ \{ A \cup B \} } =
\langle \alpha \mid \alpha^2 \rangle$.
\end{enumerate}
\end{lemma}

\begin{center}
\begin{overpic}[width=14cm, clip]{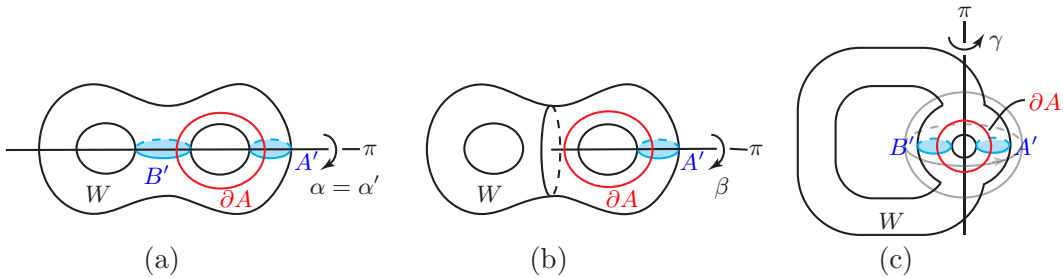}
  \linethickness{3pt}
  \put(52,0){(a)}
  \put(30,25){\footnotesize $W$}
  \put(52,32){\footnotesize {\color{blue}$B'$}}
  \put(81,23){\footnotesize {\color{red}$\partial A$}}
  \put(109,37){\footnotesize {\color{blue}$A'$}}
  \put(135,43){\footnotesize $\pi$}
  \put(115,28){\footnotesize $\alpha = \alpha'$}

  \put(198,0){(b)}
  \put(178,25){\footnotesize $W$}
  \put(227,23){\footnotesize {\color{red}$\partial A$}}
  \put(255,37){\footnotesize {\color{blue}$A'$}}
  \put(281,43){\footnotesize $\pi$}
  \put(268,28){\footnotesize $\beta$}

  \put(330,0){(c)}
  \put(330,15){\footnotesize $W$}
  \put(334,43){\footnotesize {\color{blue}$B'$}}
  \put(381,43){\footnotesize {\color{blue}$A'$}}
  \put(387,60){\footnotesize {\color{red}$\partial A$}}
  \put(360,95){\footnotesize $\pi$}
  \put(372,84){\footnotesize $\gamma$}
\end{overpic}
\captionof{figure}{Generators of $\mathcal{G}_{\{ A \}}$.}
\label{fig:isotropy_of_a_primitive_disk}
\end{center}

\begin{lemma}
\label{lem:stabilizers of vertices of TT}
\begin{enumerate}
\item
\label{item:stabilizers of vertices of TT (1)}
If $q^2 \equiv 1 \pmod p$, we have
$\mathcal{G}_{\{\mathcal{T}_1\}} =
\langle \alpha \mid \alpha^2 \rangle \oplus \langle \beta, \gamma,
\sigma_1, \sigma_2 \mid
{\gamma}^2, {\sigma_1}^2, {\sigma_2}^2 \rangle$.
\item
\label{item:stabilizers of vertices of TT (2)}
If $q^2 \not\equiv 1 \pmod p$, we have
$\mathcal{G}_{\{\mathcal{T}_1\}}
= \langle \alpha \mid \alpha^2 \rangle \oplus \langle \beta_1, \beta_2, \gamma_1, \gamma_2,
\sigma_1, \sigma_2 \mid
{\gamma_1}^2, {\gamma_2}^2, {\sigma_1}^2, {\sigma_2}^2 \rangle$.
\end{enumerate}
\end{lemma}
\begin{proof}
\noindent (\ref{item:stabilizers of vertices of TT (1)})
Suppose $q^2 \equiv 1 \pmod p$.
Since the argument is almost the same as Theorem 5.7 (2)-(c) in \cite{CK15b}, we explain the outline.
A local part of $\mathcal{T}_1$ containing vertices
$E$, $C$, $E_1$ and $C_1$ is illustrated
in Figure \ref{fig:primitive_disk_complex_3} (a).

\begin{center}
\begin{overpic}[width=14cm, clip]{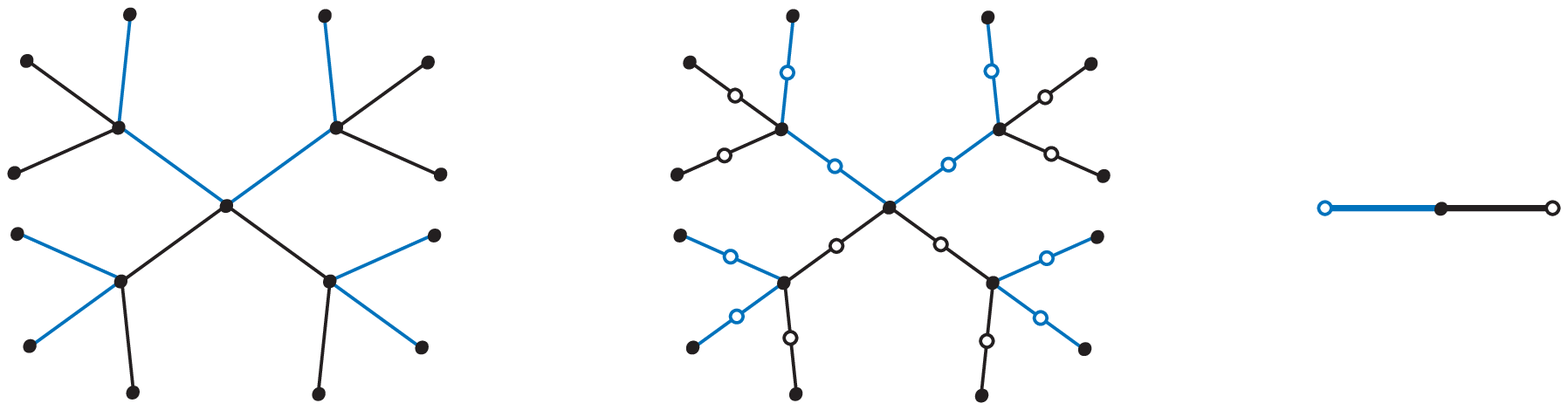}
  \linethickness{3pt}
  \put(56,0){(a)}
  \put(40,42){\small $E_1$}
  \put(-3,72){\small $C_1$}
  \put(69,64){\small $E = C_q$}
  \put(40,85){\small $C = E_q$}
  \put(65,45){\small $E_{p-1}$}
  \put(0,106){\small $C_{p-1}$}
  \put(42,111){\small $C_{p-q}$}
  \put(97,83){\small $E_{p-q}$}
  \put(218,0){(b)}
  \put(232,63){$E$}
  \put(203,85){$C$}
  \put(354,0){(c)}
  \put(315,52){$\{E, C\} $}
  \put(355,52){$E$}
  \put(373,52){$\{ E , E_1 \}$}
\end{overpic}
\captionof{figure}{(a) The tree component $\mathcal{T}_1$.
(b) The tree $\mathcal{T}_1'$. (c) The quotient
$\mathcal{T}_1' / \mathcal{G}_{\{\mathcal{T}_1\}}$.}
\label{fig:primitive_disk_complex_3}
\end{center}

Let $\mathcal{T}_1'$ be the first barycentric subdivision
of $\mathcal{T}_1$, which is described in Figure \ref{fig:primitive_disk_complex_3}  (b).

By Lemma \ref{lem:number of orbits of primitive disks and pairs},
the quotient of $\mathcal{T}_1'$ by the action of $\mathcal{G}_{\{\mathcal{T}_1\}}$
is the path graph on three vertices as illustrated
in Figure \ref{fig:primitive_disk_complex_3} (c).
By Theorem \ref{thm:theorem by Brown},
we can express $\mathcal{G}_{\{ \mathcal{T}_1 \}}$ as the following amalgamated free products:
\[ \mathcal{G}_{\{ \mathcal{T}_1\} }  =
\mathcal{G}_{\{ E \cup C  \}} *_{\mathcal{G}_{ \{ E, C \} }}
\mathcal{G}_{\{ E \}} *_{\mathcal{G}_{ \{ E , E_1 \} }}
\mathcal{G}_{\{ E \cup E_1 \}} . \]
By Lemma
\ref{lem:stabilizer of a primitive disks and pairs},
we obtain the required presentation.

\smallskip

\noindent (\ref{item:stabilizers of vertices of TT (2)})
Suppose $q^2 \not\equiv 1 \pmod p$.
In this case, the argument is almost the same as Theorem 5.7 (2)-(d) in \cite{CK15b}.
A local part of $\mathcal{T}_1$ containing vertices
$E$, $C$, $E_1$, and $C_1$ is illustrated
in Figure \ref{fig:primitive_disk_complex_4} (a).

\begin{center}
\begin{overpic}[width=14cm, clip]{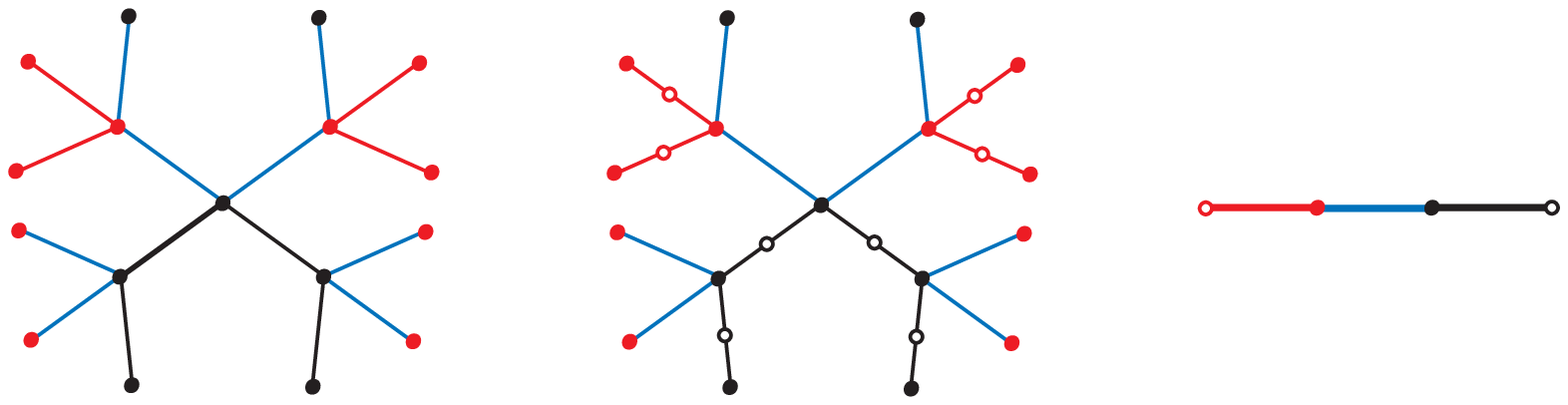}
  \linethickness{3pt}
  \put(53,0){(a)}
  \put(36,41){\small $E_1$}
  \put(-7,72){\small $C_1$}
  \put(64,63){\small $E = C_q$}
  \put(36,85){\small $C = E_{q'}$}
  \put(60,45){\small $E_{p-1}$}
  \put(-4,106){\small $C_{p-1}$}
  \put(38,111){\small $C_{p-q}$}
  \put(94,83){\small $E_{p-q'}$}
  \put(200,0){(b)}
  \put(213,62){$E$}
  \put(185,84){$C$}
  \put(339,0){(c)}
  \put(285,51){$\{ C , C_1 \}$}
  \put(326,51){$C$}
  \put(354,51){$E$}
  \put(374,51){$\{ E , E_1 \}$}
\end{overpic}
\captionof{figure}{(a) The tree component $\mathcal{T}_1$.
(b) The tree $\mathcal{T}_1'$.
(c) The quotient $\mathcal{T}_1' / \mathcal{G}_{\{ \mathcal{T}_1\} }$.}
\label{fig:primitive_disk_complex_4}
\end{center}

Let $\mathcal{T}_1'$ be the tree obtained from $\mathcal{T}_1$ by
adding the barycenter of each edge in
$\mathcal{G} \cdot \{ E, E_1 \} \cup \mathcal{G} \cdot \{ D, D_1 \}$).
See Figure \ref{fig:primitive_disk_complex_4} (b).

By Lemma \ref{lem:number of orbits of primitive disks and pairs},
the quotient of $\mathcal{T}_1'$ by the action of $\mathcal{G}_{\{ \mathcal{T}_1 \}}$
is the path graph on four vertices as illustrated in
Figure \ref{fig:primitive_disk_complex_4} (c).
By Theorem \ref{thm:theorem by Brown},
we can express $\mathcal{G}_{ \{ \mathcal{T}_1 \} }$ as the following amalgamated free products:
\[ \mathcal{G}_{\{\mathcal{T}_1\}} =
\mathcal{G}_{\{ D \cup D_1  \}} *_{\mathcal{G}_{ \{ D, D_1 \} }}
\mathcal{G}_{\{ D \}} *_{\mathcal{G}_{ \{ E , D \} }}
\mathcal{G}_{\{ E \}} *_{\mathcal{G}_{ \{ E , E_1 \} }}
\mathcal{G}_{\{ E \cup E_1 \}} . \]
By Lemma
\ref{lem:stabilizer of a primitive disks and pairs},
we obtain the required presentation.
\end{proof}

\begin{lemma}
\label{lem:stabilizers of edges of TT}
\begin{enumerate}
\item
\label{item:stabilizers of edges of TT (1)}
If $q^2 \equiv 1 \pmod p$, we have
$\mathcal{G}_{\{ \mathcal T_1 \cup \mathcal T_2 \}}
= \langle \alpha \mid \alpha^2 \rangle \oplus \langle \tau \mid {\tau}^2 \rangle$ and
$\mathcal{G}_{\{\mathcal T_1 , \mathcal T_2 \}} = \langle \alpha \mid \alpha^2 \rangle $,
where $\tau$ is an
element of $\mathcal{G}$ that exchanges $D$ and $E$.
\item
\label{item:stabilizers of edges of TT (2)}
If $q^2 \not\equiv 1 \pmod p$, we have
$\mathcal{G}_{\{ \mathcal T_1 \cup \mathcal T_2 \} } = \mathcal{G}_{\{\mathcal T_1 , \mathcal T_2 \}}
= \langle \alpha \mid \alpha^2 \rangle $.
\end{enumerate}
\end{lemma}
\begin{proof}
We first show that $\mathcal{G}_{\{\mathcal T_1 , \mathcal T_2 \}} = \langle \alpha \mid \alpha^2 \rangle $ in both cases.
Let $\varphi$ be an element of $\mathcal{G}_{\{\mathcal T_1 , \mathcal T_2 \}}$.
Since $\mathcal{B}_{ \{ D,E \} }$ is the unique bridge connecting $\mathcal{T}_1$ and $\mathcal{T}_2$,
$\varphi$ preserves $\mathcal{B}_{ \{ D,E \} }$, so $D$ and $E$.
By Lemma \ref{lem:principal pair of a bridge},
$\varphi$ preserves the shell $\mathcal{S}_{E}$.
Further, since $m + 1 < p/2$, we have $\varphi (E) = E$ and $\varphi(E_i) = E_i$
($i \in \{ 0 , 1 , \ldots, p \}$) by
Lemma \ref{lem:principal pair of a bridge}.
Let $E'$ be the unique dual disk of $E$ disjoint from $E_0$,
and let $E_0'$ be the unique semi-primitive disk disjoint from $E$.
It follows from the the uniqueness of the shell that we have
then $\varphi(E') = E'$ and $\varphi(E_0') = E_0'$.
Since
$\{E , E_{m}, E_{m+1}\}$ is a triple of pairwise
disjoint disks cutting $V$ into two 3-balls,
if $\varphi$ is orientation-preserving on $E$,
then so is on each of $E_j$, $E'$ and $E_0'$.
Then by Alexander's trick,
$\varphi$ is the trivial element of $\mathcal{G}$.
If $\varphi$ is orientation-reversing on $E$, then
so is on each of $E_j$, $E'$ and $E_0'$.
In this case, again by Alexander's trick,
$\varphi$ is the hyperelliptic involution $\alpha$.

In the remaining of the proof, we consider
the group $\mathcal{G}_{\{ \mathcal T_1 \cup \mathcal T_2 \} }$.
If $q^2 \not\equiv 1 \pmod p$,
there exists no element of $\mathcal{G}$ that
maps $E$ to $D$ by Lemma \ref{lem:primitive disks in a bridge}.
Thus, we have $\mathcal{G}_{\{ \mathcal T_1 \cup \mathcal T_2 \} } = \mathcal{G}_{\{ \mathcal T_1 , \mathcal T_2 \} }$,
which is our assertion.
If $q^2 \not\equiv 1 \pmod p$,
then by Lemmas \ref{lem:primitive disks in a bridge}
and \ref{lem:transitivity on shells}
there exists an element $\tau$ of $\mathcal{G}$ such that
$\tau(D) = E$ and $\varphi(D_i) = E_i$
($i \in \{ 0 , 1 , \ldots, p \}$).
By Lemmas \ref{lem:construction of a bridge},
\ref{lem:principal pair of a bridge} and
\ref{lem:uniqueness of a bridge},
we have $\tau(E) = D$ and $\varphi(E_i) = D_i$
($i \in \{ 0 , 1 , \ldots, p \}$).
Thus, $\tau$ is in $\mathcal{G}_{\{ \mathcal T_1 \cup \mathcal T_2 \} }$.
Now we give orientations on
$D$, $D_{m}$ and $D_{m+1}$
so that they come from an orientation of
a component $V'$ of $V$ cut off by
$D \cup D_{m} \cup D_{m+1}$.
We then can give orientations on
$E$, $E_{m}$ and $E_{m+1}$
so that they come from an orientation of
a component $\tau(V')$ of $V$ cut off by
$E \cup E_{m} \cup E_{m+1}$.
Under these orientations, both $\tau \mid_D : D \to E$ and
$\tau \mid_E : E \to D$ are oreintation-preserving.
This implies that $\tau^2 = 1 \in \mathcal{G}$.
Therefore, we have $\mathcal{G}_{ \{ \mathcal{T}_1 \cup \mathcal{T}_2 \} } =
\langle \alpha \mid \alpha^2 \rangle \oplus
\langle \tau \mid \tau^2 \rangle$.
\end{proof}

\begin{theorem}
\label{thm:presentations of the Goeritz groups for non-connected case}
Let $\mathcal G$ be the genus-$2$ Goeritz group of $L(p,q)$ with $p \not\equiv \pm 1 \pmod q$.
\begin{enumerate}
\item
\label{item:presentations of the Goeritz groups for non-connected case (1)}
If $q^2 \equiv 1 \pmod p$, the group $\mathcal{G}$ can be expressed as the amalgamated free product
$\mathcal{G}_{\{\mathcal{T}_1\}} *_{\mathcal{G}_{\{\mathcal T_1 , \mathcal T_2 \}}}
\mathcal{G}_{\{\mathcal T_1 \cup \mathcal T_2 \}}$
and it has the following presentation:
\[\langle \alpha \mid \alpha^2 \rangle \oplus \langle \beta, \gamma,
\sigma_1, \sigma_2 , \tau \mid
{\gamma}^2, {\sigma_1}^2, {\sigma_2}^2, \tau^2 \rangle.\]
\item
\label{item:presentations of the Goeritz groups for non-connected case (2)}
If $q^2 \not\equiv 1 \pmod p$, the group $\mathcal{G}$ can be expressed as the HNN-extension
$\mathcal{G}_{\{ \mathcal{T}_1 \} } *_{\mathcal{G}_{\{ \mathcal T_1 , \mathcal T_2 \} }}$
and it has the following presentation:
\[\langle \alpha \mid \alpha^2 \rangle \oplus \langle \beta_1, \beta_2, \gamma_1, \gamma_2,
\sigma_1, \sigma_2 , \upsilon \mid
{\gamma_1}^2, {\gamma_2}^2, {\sigma_1}^2, {\sigma_2}^2 \rangle .\]
\end{enumerate}
\end{theorem}
\begin{proof}
\noindent (\ref{item:presentations of the Goeritz groups for non-connected case (1)})
Let $(\TT)'$ be the first barycentric subdivision
of $\TT$.
By Lemmas \ref{lem:transitivity on the vertices of TT} and \ref{lem:stabilizers of edges of TT},
the Goeritz group $\mathcal{G}$ acts on the set of edges of $(\TT)'$
transitively, and each edge is invertible under the action of $\mathcal{G}$.
Thus,
the quotient of $(\TT)'$ by the action of $\mathcal{G}$
is a single edge with two vertices.
By Theorem \ref{thm:theorem by Brown},
we can express $\mathcal{G}$ as the following amalgamated free product:
\[ \mathcal{G} =
\mathcal{G}_{\{ \mathcal{T}_1 \}} *_{\mathcal{G}_{ \{ \mathcal{T}_1 , \mathcal{T}_2 \} }}
\mathcal{G}_{\{ \mathcal{T}_1 \cup \mathcal{T}_2 \}} . \]
By Lemmas
\ref{lem:stabilizers of vertices of TT} and \ref{lem:stabilizers of edges of TT}
we obtain the required presentation.

\noindent (\ref{item:presentations of the Goeritz groups for non-connected case (2)})
By Lemmas \ref{lem:transitivity on the vertices of TT} and \ref{lem:stabilizers of edges of TT},
the Goeritz group $\mathcal{G}$ acts on the sets of vertices and edges of $\TT$
transitively without inverting edges.
Thus, the quotient of $\TT$ by the action of $\mathcal{G}$
is a single edge with one vertex.
By Theorem \ref{thm:theorem by Brown},
we can express $\mathcal{G}$ as the following HNN-extension:
\[ \mathcal{G} =
\mathcal{G}_{\{ \mathcal{T}_1 \} } *_{\mathcal{G}_{\{ \mathcal T_1 , \mathcal T_2 \}}}. \]
By Lemmas
\ref{lem:stabilizers of vertices of TT} and \ref{lem:stabilizers of edges of TT}
we obtain the required presentation.
\end{proof}

\section{The space of Heegaard splittings}
\label{sec:The space of Heegaard splittings}

Let $M$ be a closed orientable $3$-manifold and $\Sigma$ be a Heegaard surface in $M$.
The space $\mathcal{H}(M, \Sigma)$ of Heegaard splittings equivalent to $(M, \Sigma)$
is defined to be the space of left cosets $\Diff(M)/\Diff(M,\Sigma)$,
where $\Diff(M)$ denotes the diffeomorphism group of $M$, and $\Diff(M,\Sigma)$ denotes
the subgroup of $\Diff(M)$ consisting of diffeomorphisms that preserve $\Sigma$ setwise.
The space $\mathcal{H}(M, \Sigma)$ can be regarded as the space of images of $\Sigma$ under diffeomorphisms of $M$.
Note that $\pi_0 (\mathcal{H}(M, \Sigma))$ is exactly the set of
Heegaard splittings equivalent to $(M, \Sigma)$.
By $\MCG(M)$ and $\MCG(M, \Sigma)$, we denote the groups of path components of
$\Diff(M)$ and $\Diff(M, \Sigma)$ respectively.

In \cite{JM13}	, Johnson and McCullough showed the following:
\begin{theorem}[\cite{JM13}]
\label{thm:Theorem of Johnson and McCullough}
Let $\Sigma$ be a genus-$2$ Heegaard surface of $M$.
\begin{enumerate}
\item
For each integer $q$ greater than $1$, the natural map $\pi_q(\Diff (M)) \to \pi_q (\Diff(M, \Sigma))$ is an isomorphism.
\item
There exists the following short exact sequence:
\[
1 \to \pi_1 (\Diff (M)) \to \pi_1 (\mathcal{H}(M, \Sigma)) \to G(M, \Sigma) \to 1,
\]
where $G(M, \Sigma)$ denotes the kernel of the natural map
$\MCG (M, \Sigma) \to \MCG(M)$.
\end{enumerate}
\end{theorem}
By Theorem \ref{thm:Theorem of Johnson and McCullough} (1),
for any $q \geq 2$,
the study of $\pi_q (\Diff(M, \Sigma))$ is nothing else but
that of $\pi_q (\Diff (M))$.
The following is a direct application of our result on the finite presentability of
the Goeritz groups of genus-$2$ Heegaard splittings of lens spaces.

\begin{theorem}
\label{cor:the space of Heegaard splittings of a lens space}
Let $L$ be $\mathbb{S}^3$ or a lens space $L(p,q)$, where $1 \leq q \leq p/2$.
For a genus-$2$ Heegaard surface in $L$,
$\pi_1 (\mathcal{H} (L))$ is finitely presented $($up to the Smale Conjecture for $L(2,1)$ when $L = L(2,1))$.
\end{theorem}
\begin{proof}
Since the group $\MCG(L)$ is finite for any $L$ by Bonahon \cite{Bon83},
the group $G(L, \Sigma)$ is isomorphic to
the genus-$2$ Goeritz group of $L$ up to finite extensions.
By the Smale Conjecture, proved to be correct for all $L$ except $L(2,1)$ in \cite{Hat83} and \cite{HKMR12},
$\pi_1 (\Diff (L)) = \Integer / 2 \Integer$ if $L = \mathbb{S}^3$,
$\pi_1 (\Diff (L)) = (\Integer / 2 \Integer) \oplus (\Integer / 2 \Integer)$ if $L = L(2,1)$,
$\pi_1 (\Diff (L)) = \Integer$ if $L = L(p,1)$ (for odd $p \geq 3$),
$\pi_1 (\Diff (L)) = \Integer \oplus (\Integer / 2 \Integer)$ if $L = L(p,1)$ (for even $p \geq 3$), and
$\pi_1 (\Diff (L)) = \Integer \oplus \Integer$ otherwise.
Thus, in particular, $\pi_1 (\Diff (L))$ is finitely presented for any $L$.
Now, the assertion follows from the short exact sequence in Theorem \ref{thm:Theorem of Johnson and McCullough} (2).
\end{proof}

\smallskip
\noindent {\bf Acknowledgments.}
Part of this work was carried out while the authors were visiting
National Institute for Mathematical Sciences (NIMS) in Daejeon
for the Research in CAMP Program (C-21601) and Korea Institute for Advanced Study (KIAS) in Seoul.
They are grateful to the institutes and their staff for the warm hospitality.


\begin{thebibliography}{99999}

\bibitem{Ak}Akbas, E.,
A presentation for the automorphisms of the $3$-sphere that
preserve a genus two Heegaard splitting,
Pacific J. Math. \textbf{236} (2008), no. 2, 201--222.

\bibitem{Bon83}
Bonahon, F.,
Diff\'eotopies des espaces lenticulaires, Topology {\bf 22}
(1983), no. 3, 305--314.

\bibitem{BO83}
Bonahon, F., Otal, J.-P.,
Scindements de Heegaard des espaces lenticulaires, Ann. Sci. \'Ec. Norm. Sup. (4)
{\bf 16} (1983), no. 3, 451--466.


\bibitem{C} Cho, S.,
Homeomorphisms of the $3$-sphere that preserve a Heegaard splitting of genus two,
Proc. Amer. Math. Soc. \textbf{136} (2008), no. 3, 1113--1123.

\bibitem{C2} Cho, S.,
Genus two Goeritz groups of lens spaces,
Pacific J. Math.  \textbf{265}  (2013),  no. 1, 1--16.

\bibitem{CK14} Cho, S., Koda, Y.,
The genus two Goeritz group of $\mathbb S^2 \times \mathbb S^1$,
Math. Res. Lett.  \textbf{21}  (2014),  no. 3, 449--460.

\bibitem{CK15a} Cho, S., Koda, Y.,
Disk complexes and genus two Heegaard splittings for non-prime 3-manifolds,
Int. Math. Res. Not. IMRN
{\bf 2015} (2015), 4344--4371.

\bibitem{CK15b} Cho, S., Koda, Y.,
Connected primitive disk complexes and genus two Goeritz groups of lens spaces,
Int. Math. Res. Not. IMRN
{\bf 2016} (2016), 7302-7340


\bibitem{CKA} Cho, S., Koda, Y., Seo, A.,
Arc complexes, sphere complexes and Goeritz groups,
Michigan Math. J.  \textbf{65}  (2016),  no. 2, 333--351.


\bibitem{CM09} Cho, S., McCullough, D.,
The tree of knot tunnels. Geom. Topol.  {\bf 13}  (2009),  no. 2, 769--815.

\bibitem{Ge}Goeritz, L.,
Die Abbildungen der Berzelfl\"{a}che und der Volbrezel vom Gesschlect $2$,
Abh. Math. Sem. Univ. Hamburg \textbf{9} (1933), 244--259.

\bibitem{Go} Gordon, C. McA.,
On primitive sets of loops
in the boundary of a handlebody, Topology Appl. \textbf{27}
(1987), no. 3, 285--299.


\bibitem{Hat83} Hatcher, A. E.,
A proof of the Smale conjecture, $\Diff (S^3 ) \simeq O(4)$, Ann. of Math. (2)  {\bf 117}  (1983),  no. 3, 553--607.




\bibitem{HKMR12} Hong, S., Kalliongis, J., McCullough, D., Rubinstein, J. H.,
Diffeomorphisms of elliptic $3$-manifolds.
Lecture Notes in Mathematics, 2055, Springer, Heidelberg, 2012.

\bibitem{Joh10}
Johnson, J.,
Mapping class groups of medium distance Heegaard splittings,
Proc. Amer. Math. Soc. {\bf 138} (2010), 4529--4535.

\bibitem{Joh11}
Johnson, J.,
Automorphisms of the three-torus preserving a genus-three Heegaard splitting,
Pacific J. Math.  {\bf 253}  (2011),  no. 1, 75--94.

\bibitem{JM13}
Johnson, J., McCullough, D.,
The space of Heegaard splittings,
J. Reine Angew. Math.  {\bf 679}  (2013), 155--179.

\bibitem{Kod}
Koda, Y.,
Automorphisms of the $3$-sphere that preserve spatial graphs
and handlebody-knots,
Math. Proc. Cambridge Philos. Soc. {\bf 159} (2015),
1--22.

\bibitem{McC}
McCullough, D.,
Virtually geometrically finite
mapping class groups of $3$-manifolds, J. Differential Geom.
\textbf{33} (1991), no. 1, 1--65.

\bibitem{Nam07}
Namazi, H.,
Big Heegaard distance implies finite mapping class group. Topology Appl.  {\bf 154}  (2007),  no. 16, 2939--2949.

\bibitem{OZ}
Osborne, R. P., Zieschang, H.,
Primitives in the free group on two generators,
Invent. Math. \textbf{63} (1981), no. 1, 17--24.

\bibitem{Ro} Rolfsen, D.,
\emph{Knots and links},
Mathematics Lecture Series, No. 7. Publish or Perish, Inc., Berkeley, Calif., 1976.


\bibitem{Sc} Scharlemann, M.,
Automorphisms of the $3$-sphere that preserve a genus two Heegaard splitting,
Bol. Soc. Mat. Mexicana (3) \textbf{10} (2004), Special Issue, 503--514.

\bibitem{Sc13} Scharlemann, M.,
Generating the genus $g+1$ Goeritz group of a genus $g$ handlebody,
Geometry and topology down under,  347--369, Contemp. Math., 597, Amer. Math. Soc., Providence, RI, 2013.

\bibitem{S} Serre, J.-P.,
\emph{Trees}, Translated from the French by John Stillwell,
Springer-Verlag, Berlin-New York, 1980.

















\end{thebibliography}
\end{document}